\documentclass[11pt]{amsart}
\usepackage{graphicx}
\usepackage[USenglish,english]{babel}
\usepackage[T1]{fontenc}
\usepackage[latin1]{inputenc}
\usepackage{amsfonts}
\usepackage{amssymb}
\usepackage{amsthm}
\usepackage{amsmath}
\usepackage{tikz}
\usetikzlibrary{decorations.pathreplacing}
\usetikzlibrary{shapes.misc}
\tikzset{cross/.style={cross out, draw=black, fill=none, minimum size=2*(#1-\pgflinewidth), inner sep=0pt, outer sep=0pt}, cross/.default={2pt}}
\usepackage{float}
\usepackage{enumerate}
\usepackage{accents}
\usepackage{mathtools}
\usepackage{mathrsfs}
\usepackage{comment}
\usepackage{afterpage}
\usepackage{bbm}
\usepackage{dsfont}
\usepackage{stmaryrd}
\usepackage{hyperref}
\usepackage{mleftright}
\usepackage{subfig}
\usepackage{faktor}

\theoremstyle{plain}
\newtheorem{thm}{Theorem}[section]
\newtheorem{lem}[thm]{Lemma}

\newtheorem*{thm*}{Main Theorem}
\newtheorem*{prop*}{Proposition}
\newtheorem*{cor*}{Corollary}
\newtheorem{thmintro}{Theorem}

\newtheorem{corintro}[thmintro]{Corollary}

\theoremstyle{definition}
\newtheorem{mydef}[thm]{Definition}
\newtheorem*{mydef*}{Definition}
\newtheorem{exam}[thm]{Example}
\newtheorem{rem}[thm]{Remark}
\newtheorem*{quest*}{Question}

\newtheorem*{claim*}{Claim}

\newtheorem*{IndAssum*}{Induction Assumption}
\newtheorem{assumption}[thm]{Assumptions}
\newtheorem{notation}[thm]{Notation}

\DeclareMathOperator{\Ima}{Im}
\DeclareMathOperator{\diam}{diam}

\newcommand{\CAT}{{\rm CAT(0)}}

\newcommand{\adj}{\mathrm{Adj}}

\newcounter{mcomments}

\title[Boundaries of cocompact convex sets]{Contractibility of boundaries of cocompact convex sets and embeddings of limit sets}

\author{Corey Bregman}
\address{Department of Mathematics, University of Southern Maine, Portland, ME USA}
\email{corey.bregman@maine.edu}

\author{Merlin Incerti-Medici}
\address{Universit\"at Wien, Austria}
\email{merlin.medici@gmail.com}

\begin{document}

\maketitle

\begin{abstract}
We provide sufficient conditions as to when a boundary component of a cocompact convex set in a $\CAT$-space is contractible. We then use this to study when the limit set of a quasi-convex, codimension one subgroup of a negatively curved manifold group is `wild' in the boundary. The proof is based on a notion of coarse upper curvature bounds in terms of barycenters and the careful study of interpolation in geodesic metric spaces.
\end{abstract}

\tableofcontents




\section{Introduction} \label{sec:Introduction}

Let $(X,d)$ be a $\CAT$ space and $H$ a group acting properly and freely by isometries on $X$. Let $C$ be an $H$-invariant, closed, convex subset of $X$ and $\Sigma$ a connected component of the topological boundary $\partial C \subset X$. In this paper, we study the question whether and when $\Sigma$ is contractible. This question is motivated by a particular instance in which the visual boundary of $C$, denoted $\partial_{\infty}C$, is the limit set of the action $H \curvearrowright X$, which contains information about the wildness of the action of $H$ in $X$.

\subsection{Motivation from the differentiable case} \label{subsec:motivationfromthedifferentiablecase}

Suppose the metric space $X$ above is the universal covering of an $(n+1)$-dimensional, closed, negatively curved manifold $M$ with fundamental group $G$ and suppose $H < G$ is a quasi-convex codimension one subgroup. As a $\mathrm{CAT}(-a^2)$ space, $X$ has a visual boundary, which we denote by $\partial_{\infty} X$. Since $X$ is assumed to be an $(n+1)$-dimensional manifold, we know that $\partial_{\infty} X \cong S^{n}$. We can consider the limit set
\[ \Lambda(H) := \overline{ H \cdot x } \cap \partial_{\infty}X, \]
which does not depend on the choice of $x$. The topology of the canonical embedding $\Lambda(H) \hookrightarrow \partial_{\infty} X$ contains interesting information on how wild the action of $H$ on $X$ is. Let us consider the most basic case, in which $\Lambda(H) \cong S^{n-1}$. If $n = 2$, the Jordan curve theorem tells us that the embedding $\Lambda(H) \hookrightarrow \partial_{\infty} X$ is topologically conjugate to a standard embedding of a circle into a $2$-sphere (i.e.\,the embedding of an equator). As soon as $n \geq 3$, we no longer have such a result for general embeddings $S^{n-1} \hookrightarrow S^n$, as there exist examples like Alexander's horned sphere whose complement has a component that is not simply connected. However, one might suspect that the nearly transitive actions $H \curvearrowright \Lambda(H)$ and $G \curvearrowright \partial_{\infty} X$ prevent the limit set from embedding in a wild manner. Specifically, one might ask whether, under the assumptions above, the embedding $\Lambda(H) \hookrightarrow \partial_{\infty} X$ is always conjugate by a homeomorphism to the standard embedding $S^{n-1} \hookrightarrow S^n$. Results from geometric topology tell us that this is equivalent to the statement that the two connected components of $\partial_{\infty} X \setminus \Lambda(H)$ are both contractible, except perhaps when $n=4$ (cf.\,{\cite[pg. 479]{BestvinaMess91}}).

This turns out to be a surprisingly difficult problem. In {\cite[Theorem 4.1]{ApanasovTetenov88}}, Apanasov and Tetenov construct an explicit example of a group acting by isometries on $\mathbb{H}^4$ such that its limit set is a wild sphere in $S^3$. (Specifically, the complement of the limit set consists of one contractible and one non-contractible component.) Their example does not seem to be a quasi-convex subgroup of a group that acts geometrically on $\mathbb{H}^4$, which prevents it from providing a complete negative answer to the question raised above. Nevertheless, it suggests that the study of the question, whether the limit set of a quasi-convex, codimension one subgroup is a standard embedded sphere, is likely to involve some subtle necessary and sufficient conditions.\\

Since the embedding $\Lambda(H) \hookrightarrow \partial_{\infty}X$ is standard if and only if the two connected components of $\partial_{\infty}X \setminus \Lambda(H)$ are contractible, we can study the topology of $\partial_{\infty} X \setminus \Lambda(H)$ instead of the inclusion of $\Lambda(H)$. The topology of $\partial_{\infty} X \setminus \Lambda(H)$, in turn, can be described in terms of the topological boundary of a certain convex set:

For any closed subset $S \subset \partial_{\infty} X$, we define its closed convex hull $C(S) \subset X$ as the smallest closed convex set that contains all bi-infinite geodesics between points $\xi, \eta \in S$. Since $C(S)$ is a convex subset, we can consider the set of points in $\partial_{\infty} X$ that can be represented by a geodesic ray in $C(S)$. This set yields a canonical embedding of visual boundaries $\partial_{\infty}C(S) \hookrightarrow \partial_{\infty} X$. We emphasize the distinction between the visual boundary $\partial_{\infty} C(S) \subset \partial_{\infty}X$ and the topological boundary $\partial C(S) \subset X$. A result of Anderson that states that, if $X$ has negative upper and lower curvature bounds, then for every closed subset $S \subset \partial_{\infty} X$, we have $S = \partial_{\infty} C(S)$ (see Theorem \ref{thm:Andersonresult} or {\cite[Theorem 3.3]{Anderson83}}). Since such bounds on curvature are given under our assumptions in this subsection, we put $C := C(\Lambda(H))$ and obtain $\partial_{\infty} C = \Lambda(H)$. We also define $C_{\epsilon}$ to be the closed $\epsilon$-neighbourhood of $C$, where $\epsilon > 0$. We note that $\partial_{\infty} C_{\epsilon} = \partial_{\infty} C$, but its topological boundary $\partial C_{\epsilon}$ has some useful properties that we exploit.

Since $X$ is negatively curved and $C$ is closed and convex, there exists a closest-point projection map $p_C : X \rightarrow C$. Note that $p_C$ is $H$-equivariant. This map admits an $H$-equivariant, continuous extension to $p_C : \partial_{\infty}X \setminus \Lambda(H) \rightarrow \partial C$. 
This extension provides us for every $\epsilon > 0$ with a homeomorphism $\partial_{\infty} X \setminus \Lambda(H) \rightarrow \partial C_{\epsilon}$. Contractibility of the connected components of $\partial_{\infty} X \setminus \Lambda(H)$ is thus equivalent to the contractibility of the connected components of $\partial C_{\epsilon}$. This leads us back to the central question studied in this paper: the contractibility of the components of $\partial C_{\epsilon}$. 

\subsection{Main result}

Let $X$ be a $\CAT$ space and let $H$ act properly and freely by isometries on $X$. Given two points $x, y \in X$, we denote the unique arc-length parametrised geodesic from $x$ to $y$ by $\gamma_{xy}$. Let $C \subset X$ be an $H$-invariant, closed, convex subset, $Y$ a connected component of $X \setminus C$, 
and assume that $H$ preserves $Y$. (If $X \setminus C$ has finitely many connected components, then $H$ has a finite index subgroup that preserves all connected components of $X \setminus C$. More generally, one may choose the subgroup of $H$ that preserves one connected component of $X \setminus C$.) We say that $C$ is {\it $H$-cocompact} if $H$ acts cocompactly on $C$. In order to get some better properties of the boundary of our convex set, we work with the $\epsilon$-neighbourhood $C_{\epsilon}$ rather than $C$. Let $\Sigma_{\epsilon}$ be an $H$-invariant connected component of $\partial C_{\epsilon}$, and let $Y_{\epsilon}$ be the connected component of $X \setminus C_{\epsilon}$ whose closure contains $\Sigma_{\epsilon}$. As discussed in the previous section, this is compatible with the application that we have in mind.\\

Our first key notion is the idea of a push-off grid. Since it takes several technical definitions to define a push-off grid and all its desirable properties precisely, we provide a heuristic terminology here that allows us to state our results. For the precise terminology, we refer to definitions \ref{def:pushoffgrid}, \ref{def:deltatight}, \ref{def:smalldeltarelativetoR}, and \ref{def:barycentersuptoH}. (Numbered definitions here in the introduction are precise.)

Let $\delta > 0$ and let $N_{\delta}(\cdot)$ denote the closed $\delta$-neighbourhood. A {\it $\delta$-grid} of $C_{\epsilon}$ is a discrete set of points $S \subset C_{\epsilon}$ such that
\[ N_{\delta}(S \cap \Sigma_{\epsilon}) \supset \Sigma_{\epsilon} \quad \text{and} \quad N_{\delta}(S) \supset C_{\epsilon}. \]
A {\it push-off grid of $C_{\epsilon}$ through $\Sigma_{\epsilon}$} is, roughly speaking, an $H$-equivariant map
\[ \iota : S \rightarrow Y_{\epsilon} \]
where $S$ is an $H$-invariant $\delta$-grid of $C_{\epsilon}$, such that for every $q \in S$, the geodesic $\gamma_{q\iota(q)}$ intersects $\Sigma_{\epsilon}$ in exactly one point. (If $q$ lies in the interior of $C_{\epsilon}$ this condition is always satisfied; here we use that we are working with $C_{\epsilon}$ instead of $C$. The problem lies with the $q \in S \cap \Sigma_{\epsilon}$.) We define the {\it diameter of a push-off grid} to be
\[ \diam_{intro}(\iota) := \sup\{ d( \iota(p), \iota(q) ) \vert p,q \in S, d(p,q) \leq 2\delta \}. \]


The basic idea of our construction is that, if we can find a push-off grid with some good properties, we can use this push-off grid to construct a retraction from $C$ to $\Sigma$. Doing so requires to construct a push-off grid with small diameter. This can be constructed if the image of the push-off grid sits inside a space with the following critical property.

\begin{mydef} \label{def:lambdaBarycentersIntro}
Let $Z$ be a metric space, $\lambda \in [\frac{1}{2}, 1)$, and $P \subset Z$ a finite set. We call a point $b \in Z$ a {\it $\lambda$-barycenter} of $P$ if
\[ \forall p \in P : d(b, p) \leq \lambda \cdot \diam(P). \]
If, additionally, $Q \subset Z$ is a finite subset, we call $b \in Z$ a {\it $\lambda$-barycenter of $P$ relative to $Q$} if it is a $\lambda$-barycenter of $P$ and, additionally,
\[ \forall q \in Q : d(b, q) \leq \diam( \{ q \} \cup P). \]

Let $\Delta > 0$. We say that {\it $Z$ has $\lambda$-barycenters up to diameter $\Delta$}, if for any two finite sets $P, Q \subset Z$ such that $\diam(P) \leq \Delta$ and $\diam(P \cup Q) \leq 2\Delta$, there exists a $\lambda$-barycenter of $P$ relative to $Q$.
\end{mydef}

We will discuss examples of spaces that have $\lambda$-barycenters in section \ref{subsec:Examplesandopenquestions}. For now, we simply state that injective metric spaces and $\CAT$ spaces have $\lambda$-barycenters up to any diameter and the existence of round spheres in a space puts a bound on the diameter up to which $\lambda$-barycenters can exist. In some sense, this property can thus be understood as a form of upper curvature bound.\\

We need some more terminology in order to state our results.

\begin{mydef} \label{def:diameterrelativetoHK}
    Let $\epsilon > 0$, $K \subset C_{\epsilon}$ be a compact set such that $HK = C_{\epsilon}$, and $K_{out} \subset X \setminus C_{\epsilon}$ a compact set. We define the {\it diameter of $K_{out}$ relative to $H$ and $K$} to be
    \[ \diam_K(K_{out}) := \sup \{ \diam( hK_{out} \cup K_{out}) \vert hK \cap K \neq \emptyset \}. \]
\end{mydef}


We note that we have the following bound:
\[ \diam_K(K_{out}) \leq 2( \diam(K) + \inf\{ d(p,q) \vert p \in K, q \in K_{out} \} + \diam(K_{out}) ). \]
We emphasize that the diameter of $K_{out}$ relative to $H$ and $K$ can be very different from the expression
\[ \sup\{ \diam( hK_{out} \cup K_{out}) \vert hK_{out} \cap K_{out} \neq \emptyset \} \leq 2\diam(K_{out}). \]
The following example illustrates that $\diam_K(K_{out})$ cannot be bounded in terms of the $\diam(K_{out})$.

\begin{exam} Consider the cylinder $C = D \times \mathbb{R} \subset \mathbb{R}^3$ with its boundary $\Sigma = S^1 \times \mathbb{R}$. The group $H := \faktor{ \mathbb{Z} }{ 100 \mathbb{Z} } \times \mathbb{Z}$ acts cocompactly on this cylinder and one obtains a fundamental domain whose closure looks like a sector of the disk $D$ times a bounded interval. We let $K$ be the closure of this fundamental domain and let $K_{out} = \{ p \}$ be some point in $\mathbb{R}^3$ far away from the cylinder. The action $H \curvearrowright C$ preserves the central axis of $C$ and, in particular, all elements $( [n], 0) \in H$ satisfy $([n],0) \cdot K \cap K \neq \emptyset$. Therefore, $\diam_K(K_{out}) \geq \sup\{ d(p, ([n], 0) \cdot p) \vert n \in \{0, \dots, 99\} \}$. If the distance between $p$ and $C$ is large, then the rotation around the axis of $C$ moves $p$ far away from itself, making $\diam_K(K_{out})$ large. 
\end{exam}

\begin{mydef*}[cf. Definition \ref{def:barycentersuptoH}]
    Let $\lambda \in [\frac{1}{2},1)$. We say that {\it $C_{\epsilon}$ has $\lambda$-barycenters up to $H$}, if there exist
    \begin{itemize}
        \item $\delta, \delta' > 0$ 'sufficiently small',

        \item a compact set $K$ containing a fundamental domain of $H \curvearrowright C_{\epsilon}$,

        \item an $H$-invariant $\delta$-grid $S$,

        \item a push-off grid $\iota : S \rightarrow Y_{\epsilon}$ with $\diam_{intro}(\iota) \leq \delta'$,

        \item a closed, $H$-invariant, $H$-cocompact subset $Z \subset Y_{\epsilon}$ and a compact set $K_{out}$ containing a fundamental domain of $H \curvearrowright Z$,
    \end{itemize}
    such that the following holds:
    
    \begin{enumerate}
        \item $\iota(S \cap K) \subset K_{out}$,

        \item $Z$ has $\lambda$-barycenters up to diameter $\diam_K(K_{out})$.
    \end{enumerate}
    
\end{mydef*}



\begin{thmintro} \label{thmintro:contractibilityofSigma}
Let $X$ be a $\CAT$ space, let $H$ act properly, freely, by isometries on $X$, and let $C \subset X$ be an $H$-invariant, $H$-cocompact, closed, convex subset.

Let $\epsilon > 0$, $C_{\epsilon}$ be the $\epsilon$-neighbourhood of $C$, and $\Sigma_{\epsilon}$ be an $H$-invariant connected component of $\partial C_{\epsilon}$. If there exist $\lambda \in [\frac{1}{2}, 1)$ such that $C_{\epsilon}$ has $\lambda$-barycenters up to $H$, then $\Sigma_{\epsilon}$ is contractible.
\end{thmintro}


As discussed in section \ref{subsec:motivationfromthedifferentiablecase}, we are interested in the contractibility of the connected components of $\partial C_{\epsilon}$, as this is equivalent to the contractibility of the connected components of $\partial_{\infty} X \setminus \partial_{\infty} C$, provided that $X$ is a sufficiently nice space (e.g.\,a Cartan-Hadamard manifold with curvature in $[-b^2, -1]$). We can obtain a variation of the result above that is more suitable when working with the boundary at infinity. We require the following notation.

Suppose $X$ is a Cartan-Hadamard manifold with sectional curvature in $[-b^2, -1]$, $H$ is a group acting properly, freely, and by isometries on $X$, and $C \subset X$ a closed, $H$-invariant, $H$-cocompact, convex subset. By \cite{Walter76}, the topological boundary of $C_{\epsilon}$ is a $C^{1,1}$-manifold and thus there exists a continuous normal vector field $N$ along $\partial C_{\epsilon}$ pointing outwards of $C_{\epsilon}$. Varying $\epsilon > 0$, we obtain a continuous, $H$-invariant normal vector field $N$ on $X \setminus C$ that points away from $C$. (The vector field $N$ is also the gradient of the function $d( \cdot, C)$.) We denote the flow along $N$ by $\Phi_N$ and define
\[ \Phi_N^{\infty} : X \setminus C \rightarrow \partial_{\infty} X \setminus \partial_{\infty} C, \qquad p \mapsto \lim_{t \rightarrow \infty} \Phi_N^t(p). \]
For every $\epsilon > 0$, this map defines a homeomorphism $\partial C_{\epsilon} \approx \partial_{\infty} X \setminus \partial_{\infty} C$. Our main result on contractibility of components in the visual boundary is now as follows.

\begin{thmintro} \label{thmintro:contractibilityinboundary}
    Let $X$ be a Cartan-Hadamard manifold with sectional curvature in $[-b^2, -1]$, $H$ be a group acting properly, freely, and by isometries on $X$, $C \subset X$ be a closed, $H$-invariant, $H$-cocompact convex subset, $\epsilon > 0$, $\Sigma_{\epsilon}$, an $H$-invariant connected component of $\partial C_{\epsilon}$, and $Z$ the connected component of $\partial_{\infty} X \setminus \partial_{\infty} C$ corresponding to $\Sigma_{\epsilon}$.

    Let $K \subset C_{\epsilon}$ be a compact set such that $HK = C_{\epsilon}$ and $K_{out} := \Phi_N^{\infty}( K \cap \Sigma_{\epsilon} )$, where $\Phi_n^{\infty} = \lim_{t \rightarrow \infty} \Phi_n^t$.

    If there exists $\lambda \in [\frac{1}{2}, 1)$ such that $Z$ has $\lambda$-barycenters up to diameter $\diam_K(K_{out})$, then $Z$ is contractible.
\end{thmintro}

Suppose, as in subsection \ref{subsec:motivationfromthedifferentiablecase}, that $X$ is the universal covering of an $(n+1)$-dimensional, closed, negatively curved manifold $M$ with fundamental group $G$ and suppose $H < G$ is a quasi-convex codimension one subgroup and $C = C( \Lambda(H) )$. Let $Z$ be an $H$-invariant connected component of $\partial_{\infty} X \setminus \Lambda(H)$. (If $\partial_{\infty} X \setminus \Lambda(H)$ has finitely many connected components, there is a finite index subgroup $H_0 < H$ that preserves each connected component and we can assume without loss of generality that we started with $H_0$.) Contractibility of $Z$ has some immediate consequences. Since $H$ acts freely, properly discontinuously, and cocompactly on $Z$ (see \cite{Swenson01}), contractibility of $Z$ implies that $\faktor{Z}{H}$ is a manifold $K(H, 1)$. By Bestvina-Mess \cite{BestvinaMess91}, this implies that $\Lambda(H)$ is a homology sphere. We thus obtain the following corollary.

\begin{corintro}
    Let $X$, $H$, $C$, and $C_{\epsilon}$ be as in the previous paragraph. Let $Z$ be an $H$-invariant connected component of $\partial_{\infty} X \setminus \Lambda(H)$ and suppose there exists a compact set $K \subset C_{\epsilon}$ such that $HK = C_{\epsilon}$ and such that $Z$ has $\lambda$-barycenters up to diameter $\diam_K( \Phi_N^{\infty}(K \cap \Sigma_{\epsilon}) )$. Then $\Lambda(H)$ is a homology sphere.
\end{corintro}

\subsection{Examples and open questions regarding $\lambda$-barycenters} \label{subsec:Examplesandopenquestions}

Definition \ref{def:lambdaBarycentersIntro} is a somewhat unusual definition when compared to the literature that the authors are aware of. This raises the question whether the property of having $\lambda$-barycenters can be contextualised within existing terminology. We present several crucial examples that lead to a series of questions about this property.

\begin{exam}[Model case: simplicial complexes]
    Let $S$ be a (locally finite, finite-dimensional) simplicial complex. We can equip $S$ with a metric by equipping every simplex with the euclidean metric such that all edges have length $\sqrt{2}$. (If we view an $n$-simplex as the set of solutions to the equation $\sum_{i=0}^n x_i = 1$ that satisfy $x_i \geq 0$ for all $i$, we naturally obtain this metric on the simplex.) Let $\sigma$ be a simplex in $S$, $P$ the set of vertices of $\sigma$ and $Q$ the set of vertices of $Room(\sigma)$ (see Definition \ref{def:roomofasimplex}). It is an standard exercise to see that the barycenter of $\sigma$ is a $\frac{1}{\sqrt{2}}$-barycenter of $P$ relative to $Q$.
    
    The notion of $\lambda$-barycenters is intended to generalise this feature of simplicial complexes to the situation where we have a map $\iota : S^{(0)} \rightarrow Z$ from the vertices of $S$ into a metric space $Z$. As we will see, existence of $\lambda$-barycenters allows us to extend such maps $\iota$ to the vertices of the first barycentric subdivision of $S$ in such a way that 'the diameters of simplices shrink'.
\end{exam}

\begin{exam}[Injective metric spaces]
    A metric space $(Z,d)$ is called injective, if for every isometric embedding $f : A \rightarrow B$ and any $1$-Lipschitz map $\iota : A \rightarrow Z$, there exists a $1$-Lipschitz extension $\iota' : B \rightarrow Z$ such that $\iota' \circ f = \iota$. Injective metric spaces have $\frac{1}{\sqrt{2}}$-barycenters up to any diameter, which can be seen as follows: Let $P, Q$ be two finite (disjoint) sets in $Z$. We construct a simplicial complex $S$, whose vertices are given by $P \coprod Q$ and for every $q \in Q$, the set $\{ q \} \cup P$ spans a simplex of the appropriate dimension. We put the length of every edge between vertices in $P$ to be equal to $\diam(P)$ and for all $p \in P$, $q \in Q$, we put the length of the edge from $p$ to $q$ to be equal to $\diam( \{ q \} \cup P)$. With these edge-lengths, every simplex in $S$ is isometric to a euclidean simplex. (The non-trivial part is the existence of a euclidean simplex with these edge-lengths. Our edge-lengths are chosen to ensure existence.) Note that the barycenter of the simplex $P$ remains a $\frac{1}{\sqrt{2}}$-barycenter of $P$ relative to $Q$.
    
    Let $\iota : S^{(0)} \hookrightarrow Z$ be the map sending the vertices of $S$ to their corresponding points in $Z$. Due to the way we chose the lengths of edges in $S$, this map is $1$-Lipschitz. Let $S'$ be the first barycentric subdivision of $S$ and $b$ the barycenter of the simplex in $S$ spanned by $P$. Denote by $f : S^{(0)} \hookrightarrow S'^{(0)}$ the inclusion of the vertices of $S$ into the vertices of $S'$. Injectivity of $Z$ implies that $\iota$ extends to a $1$-Lipschitz map $\iota' : S'^{(0)} \rightarrow Z$. Since barycenters in $S$ are $\frac{1}{\sqrt{2}}$-barycenters as noted above and discussed in the previous example, the fact that $\iota'$ is $1$-Lipschitz tells us that $\iota'(b)$ is a $\frac{1}{\sqrt{2}}$-barycenter of $P$ relative to $Q$. We conclude that injective metric spaces have $\frac{1}{\sqrt{2}}$-barycenters up to any diameter.
\end{exam}

\begin{exam}[$\CAT$ spaces]
Let $X$ be a $\CAT$ space and let $p_1, \dots, p_n$, $q_1, \dots, q_m \in X$. Without loss of generality, $d(p_1, p_2) = \diam( p_1, \dots, p_n)$. Let $\gamma$ be the geodesic from $p_1$ to $p_2$ and let $b$ be the midpoint of $\gamma$. Standard arguments with comparison triangles and Euclidean geometry, together with the fact that the functions $d(q_j, \gamma(t))$ are convex, show that $b$ is a $\frac{\sqrt{3}}{2}$-barycenter of $\{ p_1, \dots, p_n \}$ relative to $\{ q_1, \dots, q_m \}$. We conclude that $X$ has $\frac{\sqrt{3}}{2}$-barycenters up to any diameter. (Note that the same holds for any dense subspace of $X$.)
\end{exam}

\begin{exam}[Spheres] \label{exam:BarycentersonSpheres}
Let $Z$ be a circle of radius $r$ embedded in the euclidean plane, equipped with the metric it inherits from $\mathbb{R}^2$, and let $x, y, z$ be three equidistant points on $Z$. One easily checks that the only points $b \in Z$ that satisfy $\max(d(b,x), d(b,y), d(b,z)) \leq \diam(x,y,z)$ are the points $x,y,z$ themselves and there is no $\lambda < 1$ for which a $\lambda$-barycenter of $\{ x,y,z \}$ exists.

The diameter of any set of equidistant points on $Z$ is equal to $\sqrt{3}r$. The moment we limit ourselves to sets with diameter $\leq \Delta < \sqrt{3}r$, we can find $\lambda < 1$ such that every set $P$ with diameter $\leq \Delta$ has a $\lambda$-barycenters relative to the empty set. Note that the full condition for having $\lambda$-barycenters up to diameter $\Delta$ is still not satisfied if $\Delta$ is close to $\sqrt{3}r$, as for any two $p_1, p_2 \in Z$ whose distance is slightly smaller than $\sqrt{3}r$, we can choose $q$ to be the point antipodal to the midpoint of the arc from $p_1$ to $p_2$. One easily sees that $\{ p_1, p_2 \}$ cannot have a $\lambda$-barycenter relative to $q$ for any $\lambda < 1$. However, if $\Delta < \frac{\sqrt{3}}{2} r$, then $P \cup Q$ is contained in a sufficiently small arc of the circle so that one can take the midpoint of the shortest arc that contains all elements of $P$ as a $\frac{1}{2}$-barycenter relative to $Q$.

A similar example can be obtained by considering four equidistant points on a $2$-sphere, although we have to choose different $\lambda$ and $\Delta$. This example also generalises to higher-dimensional spheres.
\end{exam}

We observe from these examples that the existence of $\lambda$-barycenters is on one hand related to some form of extension properties and to non-positive curvature. On the other hand, the existence of round spheres creates a bound on the diameter up to which $\lambda$-barycenters exist. The smaller (i.e.\,the more positively curved) the sphere, the smaller the diameter has to be. In light of these observations, we raise the following questions.

\begin{quest*}
    Let $Z$ be a geodesic metric space that has $\lambda$-barycenters up to diameter $\Delta$. Does there exist some $\kappa > 0$, depending on $\lambda$ and $\Delta$, such that $Z$ is a $\mathrm{CAT}(\kappa)$ space?
\end{quest*}

\begin{quest*}
    Can we describe existence of $\lambda$-barycenters up to diameter $\Delta$ as a weaker form of the extension property of injective spaces? For example, can these spaces be characterised as the spaces that satisfy a variation of the Kirszbraun extension theorem (see \cite{AlexanderKapovitchPetrunin11, LangSchroeder97})?
\end{quest*}

As mentioned in section \ref{subsec:motivationfromthedifferentiablecase}, we are particularly interested in existence of $\lambda$-barycenters for codimension one submanifolds inside Cartan-Hadamard manifolds. In that context, we raise the following question.

\begin{quest*}
Let $X$ be a Cartan-Hadamard manifold with pinched negative curvature and $\Sigma$ a codimension one $C^2$-differentiable submanifold in $X$. Suppose the second fundamental form of $\Sigma$ is bounded by some constant $k$ (in the sense that the second fundamental form is bounded when applied to unit vectors). Are there $\lambda$ and $\Delta$, depending only on $k$ and the dimension of $X$, such that $\Sigma$ has strong $\lambda$-barycenters up to diameter $\Delta$?
\end{quest*}

The rest of the paper is structured as follows. In section \ref{sec:Preliminaries}, we introduce and recall some notation and basic facts about geometry of negatively curved spaces and simplicial complexes that we will need. In section \ref{sec:RetractingtotheBoundary}, we show how the retraction $C_{\epsilon} \rightarrow \Sigma_{\epsilon}$ is constructed under certain geometric assumptions and we show that the existence of $\lambda$-barycenters is sufficient for the construction of this retraction. In section \ref{sec:applications}, we apply our general construction to the specific case of cocompact, negatively curved Cartan-Hadamard manifolds and convex hulls of quasi-convex, codimension one subgroups and we prove Theorem \ref{thmintro:contractibilityinboundary}.\\

{\bf Acknowledgments:} The authors are grateful to Fanny Kassel, Alessandro Sisto, and Stefan Stadler for several helpful discussions and suggestions. The second author thanks Ivan Beschastnyi and Veronica Fantini for discussing a previous version of the retraction-construction. The first author was supported by NSF grant DMS-2052801. The second author has been funded by the SNSF grant 194996.




\section{Preliminaries} \label{sec:Preliminaries}




\subsection{Geometry in non-positive curvature} \label{subsec:GeometryNonPositiveCurvature}

This section is mostly concerned with $\CAT$ spaces. There will be some results that are specific to $\mathrm{CAT(-1)}$ spaces or Riemannian manifolds with some curvature restrictions that are necessary for the applications that we have in mind. We will highlight these extra assumptions when we arrive at these results.

Let $X$ be a proper $\CAT$ space. A {\it geodesic} in $X$ is an isometric embedding $\gamma : I \rightarrow X$ of an interval $I \subset \mathbb{R}$. (We will work exclusively with geodesics for which $I$ is closed.) If $I$ is bounded, we also call $\gamma$ a {\it geodesic segment}, if $I$ is unbounded in one direction, we call $\gamma$ a {\it geodesic ray}, and if $I = \mathbb{R}$, we call $\gamma$ a {\it bi-infinite geodesic}. We say $X$ is {\it (uniquely) geodesically complete} if every geodesic can be extended to a (unique) bi-infinite geodesic. We will frequently use the fact that, in $\CAT$ spaces, for any two geodesics $\gamma$, $\gamma'$ and any closed convex set $C$, the functions $t \mapsto d(\gamma(t), \gamma'(t))$ and $t \mapsto d(C, \gamma(t))$ are convex. Furthermore, if $C$ consists of a single point, the latter map is strictly convex. These convexity-properties imply that the geodesic between any pair of points in $X$ is unique.

A $\CAT$ space has a boundary at infinity defined as follows: Two geodesic rays $\gamma, \tilde{\gamma}$ are called {\it asymptotically equivalent} if $\lim_{\vert t \vert \rightarrow \infty} d(\gamma(t), \tilde{\gamma}(t) ) < \infty$. The {\it visual boundary} of $X$ is defined by
\[ \partial_{\infty} X := \faktor{ \{ \gamma \text{ geodesic ray in } X \} }{ \text{asymptotic equivalence} }. \]
If a geodesic ray $\gamma : [c, \infty)$ represents a point $\xi \in \partial_{\infty} X$, we also write $\gamma(\infty) = \xi$ (and we analogously define $\gamma(-\infty)$ if $\gamma$ is defined on $(-\infty, c]$.) A basic fact about proper $\CAT$ spaces is that for every $p \in X$, $\xi \in \partial_{\infty} X$, there exists a unique geodesic ray $\gamma$ with $\gamma(0) = p$ and $\gamma(\infty) = \xi$. We call this the geodesic from $p$ to $\xi$. Given two points $p, q \in X \cup \partial_{\infty} X$, we denote the unique geodesic from $p$ to $q$ by $\gamma_{pq}$.

The space $\overline{X} := X \cup \partial_{\infty} X$ can be equipped with a topology, called {\it cone topology} such that $\overline{X}$ is compact, $X \subset \overline{X}$ is open and dense, and its topology inherited as a subspace of $\overline{X}$ is the same as the initial metric topology. This topology is generated by the following open sets: We take all open sets in $X$ and for a fixed $o \in X$ and all $\epsilon > 0, R > 0, \xi \in \partial_{\infty} X$, we add
\[ U_{o, \epsilon, R}(\xi) := \{ p \in \overline{X} \vert d( \gamma_{op}(R), \gamma_{o\xi}(R)) < \epsilon \} \setminus B_R(o), \]
where $B_R(o)$ denotes the closed ball of radius $R$, centered at $o$. Going over all $\xi, \epsilon, R$, these sets, together with all open sets in $X$, form a basis for a topology, which is called the cone topology. It is a basic fact that the cone topology does not depend on the choice of the base point $o$. The restriction of this topology to $\partial_{\infty} X$ is also called cone topology, or sometimes visual topology.

For $o \in X$, $\xi, \eta \in \partial_{\infty} X$, we define the Gromov product as
\[ (\xi \vert \eta)_o := \lim_{t \rightarrow \infty} t - \frac{1}{2} d( \gamma_{o\xi}(t), \gamma_{o\eta}(t)). \]
If $X$ is a $\mathrm{CAT(-1)}$ space, the expression
\[ \rho_o(\xi, \eta) := e^{- (\xi \vert \eta)_o}, \]
defines a metric on $\partial_{\infty} X$, which depends on the base point $o$. We call $\rho_o$ a visual metric and the topology they induce on the visual boundary is the same as the one obtained from the cone topology. (Note that this fails for $\CAT$ spaces, as boundaries of $\CAT$ spaces may have distinct points whose Gromov product, with respect to any basepoint, is infinite.)\\

\begin{rem} \label{rem:continuityofangles}
    Let $o, p, q \in X$ be three mutually distinct points. The angle $\angle_o(p,q)$ is defined to be the angle at which the two geodesic segments $\gamma_{op}$ and $\gamma_{oq}$ meet at $o$. By {\cite[Chapter II.3, Proposition 3.3]{CAT(0)reference}}, the angle is continuous in $p$ and $q$ and upper semicontinuous in $o$. By {\cite[Lemma 3.3]{Shiohama93}}, the angle is lower semicontinuous in $o$, if $X$ satisfies a lower curvature bound in the sense of Alexandrov. In particular, if $X$ has both an upper and a lower curvature bound in the sense of Alexandrov, then the angle is continuous in all three arguments at every triple in which $o$ does not coincide with $p$ or $q$.
\end{rem}

Let $C \subset X$ be a closed, convex subset. Since $X$ is $\CAT$, there is a well-defined closest-point projection $\pi_C : X \rightarrow C$, which is continuous. Given a set $S \subset X$, we define $C(S)$ to be the closed convex hull of $S$, i.e.\,the smallest closed convex set in $X$ that contains $S$. Similarly, if $S \subset \overline{X}$, we define $C(S)$ to be the smallest closed convex set in $X$ that contains all (possibly bi-infinite) geodesics between points in $S$. Given a convex set $C \subset X$, we can consider all points in $\partial_{\infty} X$ that can be represented by geodesic rays contained in $C$ and refer to this as $\partial_{\infty} C$. 

Let $C \subset X$ be any subset and $R > 0$. We define the $R$-neighbourhood of $C$ by
\[ N_R(C) := \{ p \in X \vert d(p, C) \leq R \}. \]
If $C$ is a convex set and $\epsilon > 0$, we define $C_{\epsilon} := N_{\epsilon}(C)$, which is again a closed convex set since we are working in a $\CAT$ space. We are interested in the topological boundary
\[ \partial C_{\epsilon} = \{ x \in X \vert d(x, C) = \epsilon \} \subset X. \]
This set should not be confused with the visual boundary $\partial_{\infty} C_{\epsilon} = \partial_{\infty} C$ of $C_{\epsilon}$, which lies in $\partial_{\infty} X$.

\begin{rem}
    We work with $C_{\epsilon}$ rather than the initial set $C$ because of the following property: If a geodesic $\gamma$ starts in $C$ and leaves $C_{\epsilon}$, then it intersects $\partial C_{\epsilon}$ in exactly one point. Indeed, since $\gamma$ starts inside of $C$, the convex function $t \mapsto d(\gamma(t), C)$ has to be strictly increasing when it meets $\partial C_{\epsilon}$ for the first time. By convexity, this function will remain strictly increasing for all time after the first time it meets $\partial C_{\epsilon}$ and can never obtain the value $\epsilon$ again. We conclude that $\gamma$ meets $C_{\epsilon}$ in exactly one point.
\end{rem}

Let $\Sigma_{\epsilon}$ be a connected component of $\partial C_{\epsilon}$. There exists a unique connected component $Y_{\epsilon}$ of $X \setminus C_{\epsilon}$ such that $\partial Y_{\epsilon} \cap \Sigma_{\epsilon} \neq \emptyset$. (We will never use this, but it is easy to see that, if this intersection is non-empty, then $\partial Y_{\epsilon} = \Sigma_{\epsilon}$.) We call $Y_{\epsilon}$ the {\it connected component of $X \setminus C_{\epsilon}$ bounded by $\Sigma_{\epsilon}$}. Suppose $\epsilon, \epsilon' > 0$ and $\Sigma_{\epsilon}$, $\Sigma_{\epsilon'}$ be connected components of $\partial C_{\epsilon}$ and $C_{\epsilon'}$ respectively. We say that $\Sigma_{\epsilon}$ and $\Sigma_{\epsilon'}$ correspond to the same component of the complement, if the components $Y_{\epsilon}$, $Y_{\epsilon'}$, bounded by $\Sigma_{\epsilon}$ and $\Sigma_{\epsilon'}$ respectively, intersect.




\subsection{Cartan-Hadamard manifolds} \label{subsec:CartanHadamardmanifolds}

Let $X$ be a Cartan-Hadamard manifold, that is a geodesically complete, simply connected Riemannian manifold with non-positive sectional curvature. Furthermore, suppose that there exists some $b$ such that the sectional curvature in $X$ is always greater than $-b^2$. In this subsection, we cover the special properties that Cartan-Hadamard manifolds enjoy over general $\CAT$ spaces, including some results that specifically hold for Cartan-Hadamard manifolds with a negative upper curvature bound.\\

The following is a crucial question for our results in section \ref{sec:applications}: Given a closed set $S \subset \partial_{\infty} X$, is $\partial_{\infty} C(S) = S$? A remarkable result by Anderson provides a sufficient condition for a positive answer.

\begin{thm}[{\cite[Theorem 3.3]{Anderson83}}] \label{thm:Andersonresult}
    Let $X$ be a complete, simply connected Riemannian manifold with sectional curvature in $[-b^2, -1]$ for some $b^2 \geq 1$. Then for any closed subset $S \subset \partial_{\infty} X$, we have $\partial_{\infty} C(S) = S$.
\end{thm}

\begin{rem} 
    We remark that the lower curvature bound in this result can be weakened, but not completely dropped. Furthermore, this result does not hold for general $\mathrm{CAT(-1)}$ spaces. However, there are some positive results if one replaces convexity by a suitable notion of almost convexity. For more information, see \cite{HummelLangSchroeder99}.
\end{rem}

Let $\epsilon > 0$ and $C \subset X$ some convex set. By \cite{Walter76} the topological boundary of $C_{\epsilon}$ is a $C^{1,1}$-manifold. Let $\Sigma_{\epsilon}$ be a connected component of $C_{\epsilon}$ and let $Y_{\epsilon}$ be the connected component of $X \setminus C_{\epsilon}$ that is bounded by $\Sigma_{\epsilon}$. Let $N$ be the gradient field of the function $d( \cdot, C)$ on $Y_{\epsilon}$, which is the unique unit vector field that is normal to $\Sigma_{\epsilon}$ for all $\epsilon > 0$ and points away from $C$. We denote the flow along $N$ by $\Phi_N$ and define
\[ \Phi_N^{\infty} : X \setminus C \rightarrow \partial_{\infty} X \setminus \partial_{\infty} C \]
\[ p \mapsto \lim_{t \rightarrow \infty} \Phi_N^t(p). \]
For every $\epsilon > 0$, this map defines a homeomorphism $\partial C_{\epsilon} \approx \partial_{\infty} X \setminus \partial_{\infty} C$. In particular, every connected component $\Sigma_{\epsilon}$ of $\partial C_{\epsilon}$ can be canonically identified with a connected component $Z$ of $\partial_{\infty} X \setminus \Lambda(H)$, which we call the connected component corresponding to $\Sigma_{\epsilon}$.

\begin{rem}
    We point out that, if $Y_{\epsilon}$ is the connected component of $X \setminus \Sigma_{\epsilon}$ bounded by $\Sigma_{\epsilon}$, then the topological boundary of $Y_{\epsilon}$ in $\overline{X}$ with the cone topology is the union $\Sigma_{\epsilon} \cup \partial_{\infty} C \cup Z$, where $Z$ is the connected component of $\partial_{\infty} X \setminus \partial_{\infty} C$ corresponding to $\Sigma_{\epsilon}$.
\end{rem}

A subset $C \subset X$ is called {\it quasi-convex} if there exists some $R \geq 0$ such that for any two points $p, q \in C$, the geodesic $\gamma_{pq}$ is contained in $N_R(C)$. Let $X$ be a Gromov-hyperbolic Cartan-Hadamard manifold and $G$ a group acting properly and cocompactly by isometries on $X$. Let $H < G$ be a subgroup such that one (and hence all) orbits of $H$ in $X$ are quasi-convex. By \cite[Chapter III.$\Gamma$, Proposition 3.7]{CAT(0)reference}, this implies that $H$ is a hyperbolic group. We define the {\it limit set} of $H$ by
\[ \Lambda(H) = \overline{Hx_0} \cap \partial_{\infty} X, \]
where $\overline{Hx_0}$ denotes the topological closure of the orbit of $x_0$ under the action of $H$ in $\overline{X}$ equipped with the cone topology. It is well-known that the limit set does not depend on the choice of $x_0$. In our main applications, the convex set $C$ will be the convex hull of $\Lambda(H)$.

\begin{rem} \label{rem:cocompactactiononconvexhull}
    Let $C := C(\Lambda(H))$. The action of $H$ on $X$ preserves $C$ (otherwise, $hC \cap C$ would be a strictly smaller closed convex subset that contains all geodesics between points in $\Lambda(H)$). Furthermore, $H$ acts cocompactly on $C$. Indeed, if it did not, one could construct a sequence in $C$ that moves farther and farther away from an orbit $Hx_0$. This sequence admits a subsequence that converges to a point that lies in $\partial_{\infty} C = \Lambda(H)$ and in $\partial_{\infty} X \setminus \Lambda(H)$, a contradiction. (Note that we used Theorem \ref{thm:Andersonresult} in this argument and thus require its assumptions on the curvature of $X$.)
\end{rem}




\subsection{Simplicial complexes} \label{subsec:SimplicialComplexes}

In this paper, we will only be concerned with simplicial complexes, in which any set of vertices can span at most one simplex and no vertex may appear twice as a corner of the same simplex. Given a $k$-simplex with vertices $U_1, \dots, U_{k+1}$, we also write $\sigma( U_1, \dots, U_{k+1})$ for this $k$-simplex. We also say that $U_1, \dots, U_{k+1}$ span this simplex. Given the standard $k$-simplex in $\mathbb{R}^{k+1}$, it consists of all convex combinations of the form $\sum_{i = 1}^{k+1} \lambda_i e_i$, where the $e_i$ are the standard basis vectors. Analogously, we can write every point in the simplex $\sigma(U_1, \dots, U_{k+1})$ as a convex combination of the vertices $U_i$, which we write as $\sum_{i=1}^{k+1} \lambda_i U_i$.

All simplicial complexes that we will encounter are obtained in the following way.

\begin{mydef} \label{def:nerve}
    Let $X$ be a topological space and $\mathcal{U}$ a set of open sets in $X$. The {\it nerve of $\mathcal{U}$}, denoted $N(\mathcal{U})$, is the simplicial complex given by the following data:
    \begin{enumerate}
         \item The set of vertices of $N(\mathcal{U})$ is the set $\mathcal{U}$.

         \item For every collection $U_1, \dots, U_{k+1} \in \mathcal{U}$ such that $\bigcap_{i=1}^{k+1} U_i \neq \emptyset$, there is a $k$-simplex spanned by the vertices $U_1, \dots, U_{k+1}$.
    \end{enumerate}
\end{mydef}

If $\mathcal{U}$ has finite multiplicity and all elements of $\mathcal{U}$ have compact closure, then $N(\mathcal{U})$ is locally finite and the multiplicity is equal to the dimension of the largest simplex in the nerve.\\

Consider a simplicial complex $S$ with vertices $\{ U_i \vert i \in I \}$. We say a subset of indices $J \subset I$ {\it spans a simplex in $S$} if the vertices $\{ U_j \vert j \in J \}$ span a simplex in $S$ and we denote this simplex by $\sigma_J$. We will have to work with subdivisions of simplicial complexes, in particular the barycentric subdivision, which we define as follows.

\begin{mydef} \label{def:barycentricsubdivision}
    Let $S$ be a simplicial complex with vertices $\{ U_i \vert i \in I \}$. Its barycentric subdivision is the simplicial complex given by the following data:
    \begin{enumerate}
        \item The vertices are given by
        \[ \{ U_J \vert J \subset I : J \text{ spans a simplex in } S \}. \]

        \item A family of vertices $U_{J_1}, \dots, U_{j_{k+1}}$ spans a $k$-simplex if and only if $J_1 \subsetneq \dots \subsetneq J_{k+1}$, such that for all $i$, $J_i$ spans a simplex in $S$.
    \end{enumerate}
    
\end{mydef}

Note in particular that two vertices $U_J, U_{J'}$ in the barycentric subdivision are connected by an edge if and only if $J \subset J'$, or $J' \subset J$. We denote the $n$-th barycentric subdivision of $S$ by $\prescript{(n)}{}{S}$. There is a canonical homeomorphism $i_n : S \rightarrow \prescript{(n)}{}{S}$.

\begin{mydef} \label{def:roomofasimplex}
    We define the {\it room around $\sigma$}, denoted $Room(\sigma)$ to be the closure of the union of all simplices in $S$ that contain $\sigma$. We write $Room(\sigma)^{(0)}$ for the set of vertices in the room around $\sigma$.
\end{mydef}




\section{Retracting the convex hull to its boundary} \label{sec:RetractingtotheBoundary}

Let $X$ be a $\CAT$ space, $C \subset X$ a closed, convex subset, and $H$ a group acting properly and freely by isometries on $X$ such that $H$ preserves $C$ and the action of $H$ on $C$ is cocompact. For all $\epsilon > 0$, let $C_{\epsilon}$ denote the $\epsilon$-neighbourhood of $C$. Consider the topological boundary $\partial C_{\epsilon} \subset X$ of $C_{\epsilon}$. Let $\Sigma_{\epsilon}$ be a connected component of $\partial C_{\epsilon}$ and let $Y_{\epsilon} \subset X \setminus C_{\epsilon}$ be the connected component of $X \setminus C_{\epsilon}$ such that $\Sigma_{\epsilon} \subset \overline{Y_{\epsilon}}$. Since $H$ acts by isometries and preserves $C$, it also preserves $C_{\epsilon}$.

\begin{rem}
    We will assume that $H$ preserves $\Sigma_{\epsilon}$. If $X \setminus C$ has finitely many connected components, preservation of $\Sigma_{\epsilon}$ can be achieved in the following way: There is a finite index subgroup $H_0 < H$ that preserves each connected component of $\partial C_{\epsilon}$. Since $H_0$ has finite index in $H$, it has the same limit set. Since $H_0$ acts as a convergence group in its limit set, it is a quasi-convex subgroup of $H$, due to an argument of Kapovich-Kleiner \cite[Theorem 8]{KapovichKleiner00}. Since $H$ and $H_0$ have the same limit set, they also have the same convex hull $C = C(\Lambda(H)) = C(\Lambda(H_0))$ and its boundary components. Furthermore, $H_0$ acts cocompactly on $C$ as it has finite index in $H$, which acts cocompactly on $C$. We can thus replace $H$ by $H_0$ and study the same convex set, having gained the extra assumption that the action preserves each connected component of $\partial C_{\epsilon}$.
\end{rem}




\subsection{Constructing the retraction} \label{subsec:ConstructingtheRetraction}

We start with a result that will allow us to formalize a property that we need to be able to track.

\begin{mydef}
    Let $q \in \Sigma_{\epsilon}$ and $q' \in Y_{\epsilon}$. We define
    \[ \angle_q(q', C) := \angle( q', \pi_C(q) ), \]
    where $\pi_C : X \rightarrow C$ denotes the closest-point projection.    
\end{mydef}

Since $q \in \Sigma_{\epsilon}$, which does not intersect $C$, it cannot happen that $\pi_C(q) = q$ and thus the angle above is always well-defined. We need the following key-property of this angle.

\begin{lem} \label{lem:largeangleimpliesgrowingdistance}
    Let $q \in \Sigma_{\epsilon}$, $q' \in Y_{\epsilon}$, and $\gamma$ the geodesic from $q$ to $q'$. If $\angle_q(q', C) > \frac{\pi}{2}$, then $\gamma$ intersects $C_{\epsilon}$ only in $q$.
\end{lem}

\begin{proof}
    We parametrise $\gamma$ with unit speed on the interval $\gamma : [0, d(q,q')] \rightarrow X$. We need to show that for all $t > 0$ and for all $p \in C$, $d(\gamma(t), p) > \epsilon$. Since $C_{\epsilon}$ is convex, it is sufficient to find for every $p \in C$ some $\delta > 0$ such that that $d(\gamma(t), p) > \epsilon$ for all $t \in (0, \delta)$. To show this, we have to distinguish between $p$ close to $\pi_C(q)$ and $p$ far from $\pi_C(q)$.

    We know that $\angle_q(q', \pi_C(q)) > \frac{\pi}{2}$. Since $\angle_q(x, y)$ is continuous in $x$ and $y$, there exists some $\epsilon' > 0$ such that for all $p \in B_{\epsilon'}( \pi_C(q) )$, we have $\angle_q(q',p) > \frac{\pi}{2}$. By {\cite[Corollary II.3.6]{CAT(0)reference}}, this implies that $d( \gamma(t), p)$ is strictly increasing at $t = 0$. We conclude that
    \[ d(\gamma(t), p) > d(\gamma(0), p) \geq d(\gamma(0), \pi_C(q) ) = \epsilon \]
    on a small open interval $(0, \delta)$ and thus, $d(\gamma(t), p) > \epsilon$ for all $t > 0$ and all $p \in B_{\epsilon'}( \pi_C(q) )$.

    Now consider the points in $\overline{ C \setminus B_{\epsilon'}( \pi_C(q) ) }$. There exists some constant $\epsilon'' > \epsilon$ such that for all $p \in \overline{ C \setminus B_{\epsilon'}( \pi_C(q) ) }$, $d(p, q) > \epsilon''$. (This is because distances are continuous, closest points in $C$ are unique, $\overline{ C \setminus B_{\epsilon'}( \pi_C(q) ) }$ is closed, and $X$ is proper.) Let $p \in \overline{ C \setminus B_{\epsilon'}( \pi_C(q) ) }$. Since $\gamma$ is a geodesic, we have for all $t \in (0, \epsilon'' - \epsilon)$ that $d( \gamma(t), p) > \epsilon$.

    We conclude that for all $p \in C$ and all $t > 0$, $d( \gamma(t), p) > \epsilon$. Thus, $\gamma(t) \in C_{\epsilon}$ if and only if $t = 0$.
\end{proof}

Our strategy to prove contractibility of $\Sigma_{\epsilon}$ relies on working with finite covers of a fundamental domain of the action of $H$ on $C_{\epsilon}$ and extending our construction on such finite covers in an $H$-equivariant way to all of $C_{\epsilon}$. We thus have to make sure that these covers interact well with their immediate vicinity, specifically with translates of the cover by $H$ that are close to the original cover. We make a couple of definitions that will allow us to formulate this properly.

\begin{mydef} \label{def:adjacency}
    Let $K \subset X$ be a compact subset and let $\mathcal{U}$ be a finite collection of bounded open sets in $X$ that covers $K$. We define the  {\it adjacency of $\mathcal{U}$} by
\[ \adj(\mathcal{U}) := \{ h U \vert U \in \mathcal{U}, h \in H, \exists U' \in \mathcal{U} : hU \cap U' \neq \emptyset \}. \]
\end{mydef}

We emphasise that the lack of restrictions on $h$ and $U'$ implies that the adjacency of $\mathcal{U}$ contains $\mathcal{U}$ itself. Since $H$ acts properly on $X$, the finiteness of $\mathcal{U}$ extends to $\adj(\mathcal{U})$. Note that, if $U \in \adj(\mathcal{U})$ there may exist some elements $h \in H \setminus \{ 1 \}$ such that $hU \in \adj(\mathcal{U})$ as well. We will sometimes call this the `partial action' of $H$ on $\adj(\mathcal{U})$ and say that $h$ `acts' on $U$ if $hU \in \adj(\mathcal{U})$. This `partial action' extends to the nerve of $\adj(\mathcal{U})$ where we will use similar notation, whenever it makes sense.

Given a subset $K \subset X$, we will denote its orbit under $H$ by
\[ HK = \{ hx \vert h \in H, x \in K \}. \]
Furthermore, given a collection $\mathcal{U}$ of subsets of $X$, we write
\[ H \mathcal{U} := \{ hU \vert h \in H, U \in \mathcal{U} \} \quad \text{and} \quad \bigcup \mathcal{U} := \cup_{U \in \mathcal{U}} U. \]

\begin{rem} \label{rem:equivariantextensionsfromtheadjacency}
    Below, we will consider maps $\Psi : \bigcup \mathcal{U} \rightarrow N( \adj( \mathcal{U} ) )$, $\iota : N(\adj( \mathcal{U} ) ) \rightarrow Y_{\epsilon}$. Since there is a partial action of $H$ on $\bigcup \mathcal{U}$ and $N( \adj( \mathcal{U} ) )$ and $H$ acts on $Y_{\epsilon}$ (as we assume that $H$ preserves $\Sigma_{\epsilon}$), we can talk about these maps being `$H$-equivariant' in the sense that, whenever $hq \in \bigcup \mathcal{U}$, then $h \Psi(q)$ lies in $N(\adj(\mathcal{U}))$ and $\Psi(hq) = h\Psi(q)$. Similarly, if $x \in N(\adj(\mathcal{U}))$ such that $hx \in \adj(\mathcal{U})$, $H$-equivariance means simply that $\iota(hx) = h\iota(x)$. We say that $\Psi$ and $\iota$ are {\it $H$-equivariant on the adjacency} whenever the above conditions hold.

    One immediately sees that maps that are $H$-equivariant on the adjacency extend uniquely to $H$-equivariant maps $\bigcup H\mathcal{U} \rightarrow N( H \mathcal{U} )$ and $N(H \mathcal{U}) \rightarrow Y_{\epsilon}$. Thus, for all the maps we construct in this and subsequent sections, it is sufficient to construct them on the adjacency and show that they are $H$-equivariant on the adjacency. We can then uniquely extend these maps to the maps we really need.
\end{rem}

It is a standard result that there is a continuous map $\bigcup \mathcal{U} \rightarrow N(\mathcal{U})$ that sends all points in $U \in \mathcal{U}$ to simplices incident to $\mathcal{U} \in N(\mathcal{U})^{(0)}$. We need an $H$-equivariant version of this result, which we prove for the sake of completeness

\begin{lem} \label{lem:ProjectiontoNerve}
Let $X$ be a paracompact, Hausdorff topological space, $H$ a discrete group which acts properly by homeomorphisms on $X$, $K \subset X$ a closed subset and $\mathcal{U}$ a collection of open sets with compact closure in $X$ that covers $K$ and has finite multiplicity. Then there exists an open neighbourhood $\tilde{U} \subset \bigcup H \mathcal{U}$ of $HK$ and a continuous, $H$-equivariant map $\Psi : \bigcup H \mathcal{U} \rightarrow N( H \mathcal{U} )$ with the following two properties:
\begin{enumerate}
    \item For all $q \in U$ with $U \in H\mathcal{U}$, we have that $\Psi(p)$ lies in a simplex of $N( H\mathcal{U} )$ that contains $U$ as a vertex.

    \item $\Psi( \bigcup \mathcal{U} \cap \tilde{U} ) \subset N( \adj( \mathcal{U} ) )$.
\end{enumerate}
\end{lem}

\begin{proof}
We first observe that $H\mathcal{U}$ has finite multiplicity, because all elements of $\mathcal{U}$ have compact closure, $\mathcal{U}$ has finite multiplicity, and $H$ acts properly by homeomorphisms. We index $H \mathcal{U}$ with an index set $I$ and we find a subset $J \subset I$ such that $\mathcal{U} = \{ U_i \vert i \in J \}$. The action of $H$ on $H\mathcal{U}$ induces an action of $H$ on $I$ and $I = HJ$.

Since $K \subset X$ is closed and $\mathcal{U}$ is locally finite, the collection $\mathcal{U} \cup \{ X \setminus K \}$ provides a locally finite cover of $X$ by open sets. Since $X$ is paracompact and Hausdorff, there exists a continuous partition of unity subordinate to $\mathcal{U} \cup \{ X \setminus K \}$, i.e.\,continuous maps $x_i : X \rightarrow [0, 1]$ for $i \in J$ and a map $x_{X \setminus K} : X \rightarrow [0,1]$ such that $x_i\vert_{X \setminus U_i} \equiv 0$ and
\[ \sum_{i \in J} x_i + x_{X \setminus K} \equiv 1. \]
Continuity of these maps implies that there exist an open neighbourhood $U' \subset \bigcup \mathcal{U}$ of $K$ such that for all $q \in U'$, there exists some $i \in J$ with $x_i(q) > 0$. We put $\tilde{U} := HU'$, which is an open neighbourhood of $HK$.\\

We use the family $\{ x_i \vert i \in J \}$ to construct an $H$-invariant partition of unity on $\tilde{U}$ subordinate to $H\mathcal{U}$. For every $i \in J$ and every $h \in H$, we define
\[ x_{hi} : X \rightarrow [0,1] \]
\[ q \mapsto x_i(h^{-1} q). \]
Clearly, $x_{hi}$ is continuous and satisfies $x_{hi}\vert_{X \setminus hU_i} \equiv 0$. Furthermore, since for every $q \in U'$ there exists an $i \in J$ such that $x_i(q) > 0$, we find that for every $q \in \tilde{U} = HU' \supset HK$, there exists $i \in I$ such that $x_i(q) > 0$. As we have seen earlier, $H\mathcal{U}$ is locally finite, which implies that $\sum_{i \in I} x_i > 0$ is a locally finite, and thus well-defined, sum. Renormalizing, we obtain functions
\[ \overline{x_i}(q) := \frac{ x_i(q) }{ \sum_{i' \in I} x_{i'}(q) }, \]
which are well-defined on $HU$, form a continuous partition of unity subordinate to $H\mathcal{U}$, and satisfy $\overline{x_{hi}}(hq) = \overline{x_i}(q)$.\\

We now define the map $\Psi : \tilde{U} \rightarrow N(H\mathcal{U})$ by
\[ \Psi(q) := \frac{1}{\sum_{i \in I} \overline{x_i}(q) } \sum_{i \in I} \overline{x_i}(q) U_i, \]
where $U_i$ denotes the vertex in $N( H\mathcal{U} )$ induced by the element $U_i \in H \mathcal{U}$. This sum is locally finite by construction and thus, $\Psi$ is well-defined and continuous. Furthermore, $\Psi$ is $H$-equivariant since
\begin{equation*}
    \begin{split}
        \Psi(hq) & = \frac{1}{ \sum_{i \in I} \overline{x_i}(hq)} \sum_{i \in I} \overline{x_i}(hq) U_i\\
        & = \frac{1}{ \sum_{i \in I} \overline{x_i}(q)} \sum_{i \in I} \overline{x_{hi}}(hq) U_{hi}\\
        & = \frac{1}{\sum_{i \in I} \overline{x_i}(q)} \sum_{i \in I} \overline{x_i}(q) (h U_i)\\
        & = h \Psi(x),
    \end{split}
\end{equation*}
where we recall that $h$ acts on $N(H \mathcal{U})$ by sending the vertex $U_i$ to $hU_i$.

We are left to show the two properties. Let $U_{i_0} \in H\mathcal{U}$, $q \in U_{i_0}$, and $I_q$ be the set of indices such that $q \in U_i$ if and only if $i \in I_q$. By construction, $\overline{x_i}(q) = 0$ for all $i \notin I_q$. Thus, $\Psi(q)$ is a convex-combination of the vertices $U_i$ with $i \in I_q$, i.e.\,it is contained in the simplex spanned by the vertices $\{ U_i \vert i \in I_q \}$. Since $U_{i_0}$ contains $q$, we have $i_0 \in I_q$ and the first property follows. For the second statement, let $q \in \bigcup \mathcal{U}$. Then the only elements of $H \mathcal{U}$ that can contain $q$ are elements of $\adj(\mathcal{U})$. Therefore, $\Psi(q)$ is a convex combination of vertices that lie in $N( \adj(\mathcal{U}) )$. This proves the Lemma.
\end{proof}

From now on, we will work with the following assumptions and terminology.

\begin{assumption} \label{assum:Assumptions}
Let $X$ be a $\CAT$ space, $C \subset X$ be a closed convex subset, $\Sigma_{\epsilon}$ be a connected component of the topological boundary $\partial C_{\epsilon}$, $Y_{\epsilon}$ be the connected component of $X \setminus C_{\epsilon}$ bounded by $\Sigma_{\epsilon}$, and $H$ be a group acting properly and freely by isometries on $X$ such that $H$ preserves $C$ and $\Sigma_{\epsilon}$ and the action of $H$ on $C$ is cocompact.

Since the action of $H$ on $C$ is cocompact, the same is true for its action on $C_{\epsilon}$. Let $K \subset C_{\epsilon}$ be a compact set such that $HK = C_{\epsilon}$ and let $\mathcal{U}$ be a finite cover of $K$ by open sets. 
This cover induces an $H$-invariant cover $H\mathcal{U}$ of $HK = C_{\epsilon}$.
\end{assumption}

\begin{mydef} \label{def:Hfinecover}
    A cover $\mathcal{U}$ of a subset $X' \subset X$ is called {\it $H$-fine}, if for all $U \in \mathcal{U}$ and for all $h \in H \setminus \{ 1 \}$, $U \cap hU = \emptyset$.
    
\end{mydef}

We observe that, if $\mathcal{U}$ is $H$-fine, then the induced $H$-invariant cover $H\mathcal{U}$ of $HX'$ is $H$-fine as well, as $hg U \cap gU \neq \emptyset$ implies that $g^{-1} h g U \cap U \neq \emptyset$. The purpose of this definition is that, in the nerve of an $H$-fine cover, no two vertices in the same $H$-orbit can be connected by an edge. In consequence, no two simplices in the same $H$-orbit on the nerve can be faces of the same simplex. This guarantees good behaviour of the action of $H$ on $N(H\mathcal{U})$. Since $H$ acts properly and freely, one can always construct an $H$-fine cover by choosing sufficiently small open sets. (See Definition \ref{def:deltatight} and Remark \ref{rem:deltatightproperties}.)


Next, we show how a suitable map $j : N(H \mathcal{U}) \rightarrow Y_{\epsilon}$, together with the map $\Psi$ above, can be used to define a retraction $C_{\epsilon} \rightarrow \Sigma_{\epsilon}$ and how this implies contractibility of $\Sigma_{\epsilon}$. For this, we need the following definition.

\begin{mydef} \label{def:pushoff}
    A {\it $\mathcal{U}$-push-off through $\Sigma_{\epsilon}$} is a continuous map $j : N(H\mathcal{U}) \rightarrow Y_{\epsilon}$, such that there exists $\alpha_0 > \frac{\pi}{2}$ satisfying that for all $q \in \Sigma_{\epsilon}$, we have $\angle_q( \iota \circ \Psi(q), C) \geq \alpha_0$.

    We say that $C_{\epsilon}$ {\it has a push-off through $\Sigma_{\epsilon}$} if there exists a finite, $H$-fine cover $\mathcal{U}$ for some $K$ as in Assumption \ref{assum:Assumptions}, such that there exists a $\mathcal{U}$-push-off through $\Sigma_{\epsilon}$.

\end{mydef}

\begin{lem} \label{lem:pushoffimpliesretraction}
    If $C_{\epsilon}$ admits a push-off through $\Sigma_{\epsilon}$, then there exists a retraction $r : C_{\epsilon} \rightarrow \Sigma_{\epsilon}$, that is, a continuous map whose restriction to $\Sigma_{\epsilon}$ is the identity.
\end{lem}

\begin{proof}
    Let $K \subset C_{\epsilon}$ be a compact set such that $HK = C_{\epsilon}$ and $\mathcal{U}$ a finite collection of open sets in $X$ covering $K$ such that there exists a $\mathcal{U}$-push-off $j : N(H\mathcal{U}) \rightarrow Y$ through $\Sigma_{\epsilon}$. We define $r$ as follows: Let $q \in C_{\epsilon}$ and consider the geodesic $\gamma$ from $q$ to $j \circ \Psi(q)$. We claim that $\gamma$ intersects $\Sigma_{\epsilon}$ in exactly one point. Since $j \circ \Psi(q)Ê\in Y_{\epsilon}$, $\gamma$ has to leave $C_{\epsilon}$ eventually and, since $Y_{\epsilon}$ is bounded by $\Sigma_{\epsilon}$, it has to leave through $\Sigma_{\epsilon}$. Thus $\gamma$ intersects $\Sigma_{\epsilon}$ at least once.
    
    In order to show that there is exactly one point of intersection, we distinguish between two cases: If $q \notin \Sigma_{\epsilon}$, then $\gamma$ has to pass through the interior of $C_{\epsilon}$ in order to reach $\Sigma_{\epsilon}$. Thus, when $\gamma$ intersects $\Sigma_{\epsilon}$, the convex function $d(\gamma( \cdot ), C )$ is increasing at the point of intersection. Thus $\gamma$ leaves $C_{\epsilon}$ immediately after intersecting $\Sigma_{\epsilon}$ and cannot return. If $q \in \Sigma_{\epsilon}$, then the property of a push-off implies that $\angle_q( j \circ \Psi(q), C) > \frac{\pi}{2}$ and by Lemma \ref{lem:largeangleimpliesgrowingdistance}, $d(\gamma( t ), C)$ is strictly increasing at $t = 0$. We conclude that $\gamma$ immediately leaves $C_{\epsilon}$ and cannot return due to convexity. It follows that, for all $q \in C_{\epsilon}$, the geodesic $\gamma$ intersects $\Sigma_{\epsilon}$ in exactly one point. We define $r(q)$ to be this point.

    Clearly $r : C_{\epsilon} \rightarrow \Sigma_{\epsilon}$ and if $q \in \Sigma_{\epsilon}$, then $r(q) = q$. We are left to show that $r$ is continuous. Let $q_n \rightarrow q$ be a converging sequence in $C_{\epsilon}$. We denote $q'_n := j \circ \Psi(q_n)$ and $q' := j \circ \Psi(q)$. Furthermore, we denote the geodesic from $q_n$ to $q'_n$ by $\gamma_n$ and the geodesic from $q$ to $q'$ by $\gamma$. Since $\iota$ and $\Psi$ are continuous, we have that $q'_n \rightarrow q'$ and, since geodesic segments depend continuously on their endpoints in $\CAT$ spaces, $\gamma_n \rightarrow \gamma$. We set $t_0 \in \mathbb{R}$ such that $\gamma(t_0) = r(q)$.

    Let $\epsilon' > 0$. We will show that $d( r(q_n), r(q) ) < \epsilon'$ for sufficiently large $n$. Let $p_- := \gamma( t_0 - \epsilon' )$ and $p_+ := \gamma( t_0 + \epsilon' )$ and set $\epsilon_- := d( p_-, \Sigma_{\epsilon} )$ and $\epsilon_+ := d( p_+, \Sigma_{\epsilon})$. Since $\gamma$ is a geodesic that intersects $\Sigma_{\epsilon}$ in exactly one point, which is $\gamma(t_0)$, these two numbers are positive and at most $\epsilon'$. Choose $N$ sufficiently large, such that for all $n \geq N$, there exist points $p_{n,-}, p_{n,+} \in \gamma_n$ such that $d(p_{n,-}, p_-) < \epsilon_-$ and $d(p_{n,+}, p_+) < \epsilon_+$. We conclude that $p_{n,-} \in \text{Int}(C_{\epsilon})$ and $p_{n,+} \in X \setminus C_{\epsilon}$. In particular, $r(q_n)$ has to lie on $\gamma_n$ between the two points $p_{n,-}$ and $p_{n,+}$. Since distance functions are convex in $\CAT$ spaces, we conclude that
    \[ d( r(q_n), r(q) ) \leq \max( d( p_{n,-}, p_- ), d( p_{n,+}, p_+ ) ) \leq \epsilon' \]
    for all $n \geq N$. It follows that $r(q_n) \rightarrow r(q)$ and $r$ is sequentially continuous. Since we are in a metric space, this implies that $r$ is continuous.    
\end{proof}

\begin{lem} \label{lem:retractionimpliescontractibility}
    Retracts of convex sets are contractible.
\end{lem}

\begin{proof}
    Let $C$ be a convex set, $\Sigma \subset C$ and $r : C \rightarrow \Sigma$ a retraction. Since $C$ is convex, it is contractible and thus there exists a homotopy $h : C \times [0,1] \rightarrow C$ between the identity and a constant map. We restrict $h$ to $\Sigma \times [0,1]$ and compose with the retraction $r$ to obtain a map $h' = r \circ h : \Sigma \times [0,1] \rightarrow \Sigma$, which is a homotopy between the identity and a constant map. It follows that $\Sigma$ is contractible.
\end{proof}

Lemma \ref{lem:pushoffimpliesretraction} and Lemma \ref{lem:retractionimpliescontractibility} leave us with the task of constructing a $\mathcal{U}$-push-off for a suitable cover $\mathcal{U}$. We will do so by constructing the push-off on the vertices of $N( H \mathcal{U} )$ and extending it from there to the entire simplicial complex.

To define an appropriate map on the vertices of $N( H \mathcal{U} )$, we need some notation. We define $\mathcal{U}_{\Sigma} := \{ U \in \mathcal{U} \vert U \cap \Sigma_{\epsilon} \neq \emptyset \}$ and $H \mathcal{U}_{\Sigma}$ to be the induced $H$-invariant cover. We point out that $\mathcal{U}_{\Sigma}$ and $H\mathcal{U}_{\Sigma}$ are both locally finite and cover $K \cap \Sigma_{\epsilon}$ and $\Sigma_{\epsilon}$ respectively. We recall that every element $U \in H\mathcal{U}$ induces a vertex in $N(H\mathcal{U})$ and if we have a map $\iota : N(H \mathcal{U})^{(0)} \rightarrow Y_{\epsilon}$ on the $0$-skeleton of the nerve, we can consider the convex hull $C( \iota(U_1), \dots, \iota(U_{k+1}))$ of the image of $k+1$ many vertices.

\begin{mydef} \label{def:pushoffgrid}
        Let $\alpha > \frac{\pi}{2}$ and $\mathcal{U}$ a finite, $H$-fine cover of a compact set $K$ that contains a fundamental domain of $H \curvearrowright C_{\epsilon}$. A {\it $(\mathcal{U}, \alpha)$-push-off grid} is an $H$-equivariant map $\iota : N( H\mathcal{U} )^{(0)} \rightarrow Y_{\epsilon}$ such that for all $U \in \mathcal{U}_{\Sigma}$ there exists some $q \in U$ such that 
        \[ \angle_q( \iota(U) , C) \geq \alpha. \]
\end{mydef}

The key obstacle to extending a push-off grid to an actual push-off is that we need to control the angle $\angle_q( \iota \circ \Psi(q), C)$ for such an extension. To formulate a sufficient condition to do so, we need some terminology.

\begin{mydef} \label{def:pushoffdistance}
    Let $\mathcal{U}$ be a finite, $H$-fine cover of a compact set $K$ that contains a fundamental domain of $H \curvearrowright C_{\epsilon}$. Let $\iota : N(H\mathcal{U})^{(0)} \rightarrow Y_{\epsilon}$ be an $H$-equivariant map. We define the {\it push-off distance} of $\iota$ to be
    \[ d_{push-off}(\iota) := \inf \{ d( \iota(U), C_{\epsilon}) \vert U \in N(H\mathcal{U})^{(0)}  \}. \]
\end{mydef}

Since $\mathcal{U}$ is finite, $\iota$ is $H$-equivariant, and $H$ acts by isometries, the minimum in the expression above is attained. Furthermore, since $\iota(U) \notin C_{\epsilon}$ for all $U \in \mathcal{U}$, the push-off distance is always positive.

\begin{mydef} \label{def:diameterofpushoffgrid}
    Let $S$ be a simplicial complex and $\iota : S^{(0)} \rightarrow X$ be a map. We define the {\it diameter of $\iota$} to be
    \[ \diam(\iota) := \max \{ \diam( \iota( \sigma^{(0)} ) ) \vert \sigma \text{ a simplex in } S \}. \]
\end{mydef}

\begin{mydef} \label{def:deltatight}
    Let $\mathcal{U}$ be a finite, $H$-fine cover of a compact set $K$ that satisfies $HK = C_{\epsilon}$. Let $\delta, \delta' > 0$ and $\iota : N(H\mathcal{U})^{(0)} \rightarrow Y_{\epsilon}$ an $H$-equivariant map.

    We say that $\iota$ is {\it $(\delta, \delta')$-tight}, if the following two conditions hold:
    \begin{enumerate}
        \item For all $U \in H\mathcal{U}$, $\diam(U) \leq \delta$.

        \item $\diam(\iota) \leq \delta'$
    \end{enumerate}

    We say that $\iota$ is {\it $(\delta, \delta')$-tight on the boundary}, if condition (2) only applies to the restriction of $\iota$ to the subcomplex $N(H \mathcal{U}_{\Sigma}) \subset N(H\mathcal{U})$.
\end{mydef}

\begin{rem} \label{rem:deltatightproperties}
    There are a couple of immediate remarks to make about this definition.
    \begin{enumerate}
        \item[$\bullet$] Since $\iota$ is $H$-equivariant and $H$ acts by isometries, it is sufficient to check the two conditions on $N(\adj(\mathcal{U}))$. Furthermore, condition (2) can be checked entirely on the $1$-skeleton of $N(\adj(\mathcal{U}))$.

        \item[$\bullet$] Since $H$ acts properly, freely, and cocompactly on $X$, there exists $\delta > 0$ such that any $H$-invariant cover $H\mathcal{U}$, whose elements have diameter at most $\delta$, is $H$-fine. Thus, we can enforce the condition that our cover has to be $H$-fine by choosing $\delta$ sufficiently small, which is what we are going to do from now on.

        
    \end{enumerate}
\end{rem}

For the next definition, we recall that we defined $N_R(S) := \{ x \in X \vert d(x,S) \leq R$.

\begin{mydef} \label{def:smalldeltarelativetoR}
    Let $\alpha > \frac{\pi}{2}$, $K \subset C_{\epsilon}$ a compact subset such that $HK = C_{\epsilon}$, and $K_{out} \subset Y_{\epsilon}$ a compact set. We say that {\it $(\delta, \delta')$ are small relative to $K$, $K_{out}$, and $\alpha$} if the following two properties hold:
    \begin{enumerate}
        \item If $hK \cap K = \emptyset$, then $\inf\{ d(p, q) \vert p \in K, q \in hK \} > 2\delta$.
        
        \item $N_{\delta'}(K_{out}) \cap C_{\epsilon} = \emptyset$

        \item For all $q_1, q_2 \in N_{\delta}(K) \cap \Sigma_{\epsilon}$ and all $q'_1, q'_2 \in N_{\delta'}(K_{out})$, we have
    \[ d(q_1, q_2) \leq \delta, d(q'_1, q'_2) \leq \delta' \Rightarrow \vert \angle_{q_1}(q'_1, C) - \angle_{q_2}(q'_2, C) \vert \leq \frac{\alpha}{2} - \frac{\pi}{4}. \]    
    \end{enumerate}  
\end{mydef}

Since $X$ is proper and $H$ acts properly on $X$, one can always choose $\delta > 0$ sufficiently small so that condition (1) is satisfied. 

\begin{rem}
    If $X$ has an upper and a lower curvature bound, Remark \ref{rem:continuityofangles} implies that for any triple $K$, $K_{out}$, $\alpha$ there exists a pair $(\delta, \delta')$ that is small relative to $K$, $K_{out}$, and $\alpha$.
\end{rem}

This terminology allows us to formulate the following Lemma.

\begin{lem} \label{lem:extendingpushoffgridtopushoff}
    Let $\alpha > \frac{\pi}{2}$, $\mathcal{U}$ be a finite, $H$-fine cover of a compact set $K$ such that $HK = C_{\epsilon}$, and $\iota$ be a $(\delta, \delta')$-tight, $(\mathcal{U}, \alpha)$-push-off grid. Suppose $(\delta, \delta')$ are small relative to $K$, $\iota( \adj(\mathcal{U} ) )$, and $\alpha$. Then $\iota$ can be extended to a $\mathcal{U}$-push-off.
\end{lem}

In order to produce tight push-off grids for small $(\delta, \delta')$, we will subdivide our simplicial complexes in a way that reduces the size of every simplex in a uniform way. Unfortunately, to work with these subdivisions, we will need a more general and technical version of the Lemma above. We now provide the necessary terminology for subdivisions and then state and prove the more technical version of this Lemma.

\begin{mydef} \label{def:partialsubdivisions}
    Let $S$ be a simplicial complex. A {\it subdivision} $S_{sub}$ of $S$ is called {\it a locally finite subdivision} if every simplex of $S$ is subdivided into finitely many simplices.
\end{mydef}


\begin{mydef} \label{def:shrinkingsubdivision}
    Let $(X,d)$ be a metric space, $S$ be a locally finite simplicial complex, $\iota : S^{(0)} \rightarrow X$ be a 
    map and $\frac{1}{2} \leq \lambda < 1$.
    
    Let $S_{sub}$ be a locally finite subdivision of $S$ and $\iota_{1} : S_{sub}^{(0)} \rightarrow X$ an extension of $\iota$. We call the pair $(S_{sub}, \iota_1)$ a {\it $\lambda$-shrinking subdivision of $(S, \iota)$} if the following two properties hold:
    \begin{enumerate}
        \item If $\sigma'$ is a simplex in $S_{sub}$ and $\sigma$ the unique least-dimensional simplex in $S$ containing $\sigma'$, then
        \[ \diam\left( \iota\left( \sigma'^{(0)} \right) \right) \leq \lambda \diam\left( \iota\left( \sigma^{(0)} \right) \right). \]

        \item For every simplex $\sigma$ in $S$, we have that
        \[ \diam\left( \iota\left( \sigma_{sub}^{(0)} \right) \right) \leq \diam\left( \iota\left( \sigma^{(0)} \right) \right), \]
        where $\sigma_{sub}^{(0)} = \sigma \cap S_{sub}^{(0)}$ denotes the set of all vertices of $S_{sub}$ that lie in the (closed) simplex $\sigma$.
    \end{enumerate}
\end{mydef}

We will need to successively subdivide the same simplicial complex more and more, which motivates the following definition.

\begin{mydef} \label{def:higherordersubdivisions}
    If $(S_{sub}, \iota_1)$ is a $\lambda$-shrinking subdivision of $(S, \iota)$, we say it is of {\it order 1}. We inductively define a {\it $\lambda$-shrinking subdivision of order $n$} to be a $\lambda$-shrinking subdivision of a $\lambda$-shrinking subdivision of order $n-1$.
\end{mydef}

\begin{lem} \label{lem:Diametersinshrinkingsubdivisions}
    Let $(S_n, \iota_n)$ be a $\lambda$-shrinking subdivision of $(S, \iota)$ of order $n$. 
    \begin{enumerate}
        \item For every simplex $\sigma'$ in $S_n$, we have
        \[ \diam\left( \iota\left( \sigma'^{(0)} \right) \right) \leq \lambda^n \diam(\iota). \]

        \item For every vertex $v \in S_n^{(0)}$, let $\sigma$ be the least-dimensional (closed) simplex in $S$ containing $v$. Then every vertex $v_0 \in \sigma^{(0)}$ satisfies
        \[ d( \iota_n(v_0), \iota_n(v) ) \leq \sum_{i=0}^{n-1} \lambda^n \diam\left(\iota\left(\sigma^{(0)}\right)\right) \leq \frac{ \diam\left(\iota\left(\sigma^{(0)} \right) \right) }{1-\lambda}.\]
        In particular, if $\sigma_n^{(0)}$ denotes the set of vertices in $S_n$ that lie in $\sigma$, then
        \[ \diam\left( \iota\left( \sigma_n^{(0)} \right) \right) \leq \frac{2}{1-\lambda} \diam\left( \iota\left( \sigma^{(0)} \right) \right) \leq \frac{2}{1-\lambda} \diam(\iota). \]

        \item The images of $\iota$ and $\iota_n$ satisfy
        \[ \iota( S_n^{(0)} ) \subset N_{ \frac{ \diam(\iota) }{1-\lambda} }\left(\iota\left(S^{(0)} \right) \right). \]
    \end{enumerate}
\end{lem}

\begin{proof}
    Statement (1) is an immediate consequence of the definition of $\lambda$-shrinking subdivisions. For statement (2), let $(S_1, \iota_1), \dots, (S_n, \iota_n)$ be the sequence of $\lambda$-shrinking subdivisions of increasing order that subdivides $S$ into $S_n$ in $n$ many steps. Let $v$ be a vertex in $S_n$ and $\sigma$ be the least-dimensional (closed) simplex in $S$ containing $v$. Let $\sigma_{n-1}$ the least-dimensional (closed) simplex in $S_{n-1}$ containing $v$. Choose a vertex $v_{n-1} \in S_{n-1}^{(0)}$ of $\sigma_{n-1}$. By property (2) in the definition of a $\lambda$-shrinking subdivision, we know that
    \[ d( \iota_n(v_{n-1}), \iota_n(v) ) \leq \diam( \iota( \sigma_{n-1}^{(0)} ) ) \leq \lambda^{n-1} \diam(\iota( \sigma^{(0)} ) ). \]
    We inductively obtain a sequence of simplices $\sigma_{n-1}, \sigma_{n-2}, \dots, \sigma_0 = \sigma$ and a sequence of vertices $v = v_n, v_{n-1}, v_{n-2}, \dots, v_0$ of vertices such that $v_i$ is a vertex of $\sigma_i$, which is a simplex of the subdivision $S_i$. These vertices satisfy the inequalities
    \[ d( \iota_n( v_i ), \iota_n( v_{i+1} ) ) \leq \diam( \iota_i( \sigma_i^{(0)} ) ) \leq \lambda^{i} \diam( \iota( \sigma^{(0)} ) ). \]
    Therefore, we estimate
    \[ d( \iota_n( v ), \iota_n( v_0 ) ) \leq \sum_{i = 0}^{n-1} \lambda^{i} \diam( \iota( \sigma^{(0)} ) ) \leq \frac{ 1 }{(1-\lambda)} \diam( \iota( \sigma^{(0)} ) ). \]
    Since $\sigma_0 = \sigma$ and the vertex $v_0$ could be chosen to be any vertex of $\sigma_0$, we obtain the estimate above for every $v_0 \in \sigma^{(0)}$. This proves statement (2). Statement (3) is an immediate consequence of statement (2).
\end{proof}

\begin{lem} \label{lem:extendingpushoffgridtopushoff}
    Let $\lambda \in [\frac{1}{2},1)$, $\alpha > \frac{\pi}{2}$, $\mathcal{U}$ be a finite, $H$-fine cover of a compact set $K$ such that $HK = C_{\epsilon}$, and $\iota$ a $(\mathcal{U}, \alpha)$-push-off grid which is $(\delta, (1-\lambda) \delta' )$-tight on the boundary.
    
    Suppose $(\delta, \delta')$ is small relative to $K$, $\iota(\adj(\mathcal{U}) )$, and $\alpha$ and suppose $\iota$ admits an $H$-equivariant $\lambda$-shrinking subdivision $(S_n, \iota_n)$ of order $n$ which is $(\delta, \delta')$-tight, has push-off distance $> \delta'$, and satisfies $\Ima(\iota_n) \subset Y_{\epsilon}$. Then $\iota_n$ can be extended to a $\mathcal{U}$-push-off through $\Sigma_{\epsilon}$.
\end{lem}

\begin{proof}
    Let $K_{out} := N_{\delta'}( \iota( \adj(\mathcal{U} ) ) )$. Our strategy is to construct an $H$-equivariant map $j : N( H(\mathcal{U}) ) \rightarrow N_{\delta'}(HK_{out})$, which satisfies the necessary condition on angles to be a push-off of $C_{\epsilon}$ through $\Sigma_{\epsilon}$. 

    Let $(S_n, \iota_n)$ be a $(\delta, \delta')$-tight, $\lambda$-shrinking subdivision of $N( H \mathcal{U} )$. We first define $j$ on the vertices of $S_n$ by setting $j(U) := \iota_n(U)$ for every vertex $U \in S_n$. Since $\iota_n$ is $H$-equivariant, $j$ is $H$-equivariant where we have defined it. We now use induction to extend $j$ to the $(l+1)$-skeleton of $N( H \mathcal{U} )$ for all $l \geq 0$.

    \begin{IndAssum*}
        Suppose $j$ has been defined on a subcomplex of $S_n$ that contains the $l$-skeleton of the complex and possibly some $(l+1)$-simplices such that, if $j$ is defined on an $(l+1)$-simplex $\sigma$, then it is defined on every translate $h\sigma$ in $S_n$, and $j$ is $H$-equivariant wherever it is defined. Furthermore, suppose that for all simplices $\sigma(U_1, \dots, U_{k+1})$ on which $j$ is already defined, we have $j( \sigma(U_1, \dots, U_{k+1}) ) \subset C( \iota(U_1), \dots, \iota(U_{k+1}) )$.
    \end{IndAssum*}

    Clearly, the induction assumption is satisfied on the $0$-skeleton with how we defined $j$ above. Now suppose the induction assumption holds for some $l$. Let $\sigma = \sigma(U_1, \dots, U_{l+2})$ be a simplex in $N(H\mathcal{U})$ on whose interior $j$ has not been defined yet. Since $j$ is defined on the $l$-skeleton, we see that $j$ is defined on $\partial \sigma$ and $j(\partial \sigma) \subset C( \iota(U_1), \dots, \iota(U_{l+2}) )$. Therefore, the restriction of $j$ to $\partial \sigma$ defines an element of the $l$-th homotopy group $\pi_l(C(\iota(U_1), \dots, \iota(U_{l+2})))$. Since $C(\iota(U_1), \dots, \iota(U_{l+2}))$ is convex, it is contractible and there exists a continuous extension of $j$ to $\sigma$ that maps $\sigma$ into $C(\iota(U_1), \dots, \iota(U_{l+2}))$. We extend $j$ to $\sigma$ in this way and, additionally, extend $j$ to $h\sigma$ for every $h \in H$ by the formula $j(hx) := hj(x)$ for every $x \in \sigma$.

    We need to show that this extension of $j$ is well-defined. There are three concerns:
    \begin{enumerate}
        \item $j$ could have been defined on $h\sigma$ already.
    
        \item The formula above provides a definition for $j$ on $\partial (h\sigma)$, where it was already defined.

        \item If $\sigma$ and $h\sigma$ share some face, there may be conflicting definitions occurring.
    \end{enumerate}
    The first concern is taken care of since we assume by induction that, whenever $j$ is defined on $\sigma$, it is defined on $h\sigma$ for all $h \in H$. The second concern is adressed by the fact that $j$ is $H$-equivariant on the $l$-skeleton by the induction assumption, which means that the new definition of $j$ on $\partial h\sigma$ coincides with the old one. Finally, the last concern is alleviated by the fact that $\sigma$ and $h\sigma$ cannot share a face by our assumption that $\mathcal{U}$ is $H$-fine. 
    Thus, $j$ extends continuously to $\sigma$ and all its translates.

    Since $N(H\mathcal{U})$ is a cocompact simplicial complex, this procedure allows us to define in finitely many steps a continuous map $j: N(H\mathcal{U}) \rightarrow C_{N(H\mathcal{U}))}(\iota)$ that is $H$-equivariant.\\
    
    We are left to show that $j$ is a $\mathcal{U}$-push-off. We first show that the image of $j$ is contained in $Y_{\epsilon}$. By construction, every simplex $\sigma$ in $S_n$ satisfies
    \[ j(\sigma^{(0)}) \subset C(\iota_n( \sigma^{(0)} )). \]
    Since $\iota_n$ is $(\delta, \delta')$-tight, we can continue this inclusion with
    \[ C(\iota_n( \sigma^{(0)} )) \subset N_{\delta'}( \iota_n( \sigma^{(0)} ) ). \]
    Since $\iota_n$ has push-off distance strictly greater than $\delta'$ by assumption, we conclude that $C( \iota_n( \sigma^{(0)} ) ) \cap C_{\epsilon} = \emptyset$ for every simplex in $S_n$, which proves the first property required of a $\mathcal{U}$-push-off.\\

    We are left to prove that $j$ satisfies for all $q \in \Sigma_{\epsilon}$ that $\angle_q( j \circ \Psi(q), C) > \frac{\pi}{2}$. Since this angle is invariant under the action of $H$, we can assume without loss of generality that $q \in \Sigma_{\epsilon} \cap K$. Let $U \in \mathcal{U}_{\Sigma}$ such that $q \in U$. Since $\iota$ is a $(\mathcal{U}, \alpha)$-push-out grid, there exists some $q_0 \in U$ such that $\angle_q(\iota(U), C) \geq \alpha$. Since $\iota$ is $(\delta, \delta')$-tight, we know that $d(q, q_0) \leq \delta$.

    We now estimate $d( j \circ \Psi(q), \iota(U) )$. By construction of $\Psi$ in Lemma \ref{lem:ProjectiontoNerve}, we know that $\Psi(q)$ lies in a simplex $\sigma$ that has $U$ as a vertex and, since $q \in \Sigma_{\epsilon}$, $\sigma$ is a simplex of $N(H\mathcal{U}_{\Sigma})$. Moving to the subdivision, $\Psi(q)$ lies in a simplex $\sigma_n$ of $S_n$ that lies in $\sigma$. Let $U_1, \dots, U_{k+1}$ be the vertices of $\sigma_n$. By construction of $j$, we know that $j \circ \Psi(q) \in C( \iota( U_1 ), \dots, \iota(U_{k+1} ) )$. By statements (1) and (2) of Lemma \ref{lem:Diametersinshrinkingsubdivisions}, we know that
    \begin{equation*}
        \begin{split}
            d( \iota(U), j \circ \Psi(q) ) & \leq d( \iota(U), \iota_n(U_1) ) + d(\iota_n(U_1), j \circ \Psi(q) )\\
            & \leq \sum_{i=0}^{n-1} \lambda^i \diam\left( \iota\left( \sigma^{(0)} \right) \right) + \lambda^n \diam\left( \iota\left( \sigma^{(0)} \right) \right)\\
            & \leq \frac{1}{1-\lambda} \diam \left( \iota \left( \sigma^{(0)} \right) \right).
        \end{split}
    \end{equation*}
    Since $\iota$ is $(\delta, (1-\lambda)\delta')$-tight on the boundary and all vertices of $\sigma$ lie in $H\mathcal{U}_{\Sigma}$, we know that $\frac{1}{1-\lambda}\diam( \iota( \sigma^{(0)} ) ) \leq \delta'$. We conclude that
    \[ d( \iota(U), j \circ \Psi(q) ) \leq \delta'. \]
    We see that $j$ sends the subcomplex $N( H \mathcal{U}_{\Sigma} ) \cap N( \adj( \mathcal{U} ) )$ into the $\delta'$-neighbourhood of $\iota( \adj(\mathcal{U} ) )$. Since $(\delta, \delta')$ are small relative to $K$, $\iota( \adj(\mathcal{U} ) )$, and $\alpha$, this implies that
    \[ \angle_q(j \circ \Psi(q), C) \geq \angle_{q_0}( \iota(U), C ) - \left( \frac{\alpha}{2} - \frac{\pi}{4} \right) \geq \frac{\alpha}{2} + \frac{\pi}{4} > \frac{\pi}{2}. \]
    We thus obtain the required inequality for all $q \in \Sigma_{\epsilon} \cap K$ and, by $H$-equivariance, for all $q \in \Sigma_{\epsilon}$. This proves that $j$ is a $\mathcal{U}$-push-off of $C_{\epsilon}$ through $\Sigma_{\epsilon}$, which proves the Lemma.
\end{proof}

We summarize the results of this section in the following theorem.

\begin{thm} \label{thm:Tightpushoffgridsimplycontractibility}
    Let $\epsilon > 0$, $X$ be a $\CAT$ space with an isometric action by a group $H$, $C \subset X$ be convex, $H$-invariant, and $H$-cocompact, $\Sigma_{\epsilon}$ be an $H$-invariant boundary component of $C_{\epsilon}$, and $Y_{\epsilon} \subset X \setminus C_{\epsilon}$ the connected component bounded by $\Sigma_{\epsilon}$.

    Let $K \subset C_{\epsilon}$ be a compact set such that $HK = C_{\epsilon}$, $K_{out} \subset Y_{\epsilon}$ a compact set, $\alpha > \frac{\pi}{2}$, $\mathcal{U}$ a finite cover of $K$ by open sets, and $(\delta, \delta')$ small relative to $K, K_{out}$, and $\alpha$. Then the following holds:

    \begin{itemize}
        \item If there exists a $(\delta, \delta')$-tight $(\mathcal{U}, \alpha)$-push-off grid $\iota$ such that $\iota( \mathcal{U} ) \subset K_{out}$, then $\Sigma_{\epsilon}$ is contractible.

        \item If there exists a $(\mathcal{U}, \alpha)$-push-out grid $\iota$ that is $(\delta, (1-\lambda)\delta')$-tight on the boundary, that satisfies $\iota( \mathcal{U} ) \subset K_{out}$, and that admits a $\lambda$-shrinking subdivision $(S_n, \iota_n)$ of order $n$ such that $\iota_n$ is $(\delta, \delta')$-tight and satisfies $\delta' < d_{push-off}(\iota_n)$, then $\Sigma_{\epsilon}$ is contractible.

    \end{itemize}

\end{thm}




\subsection{Shrinking subdivisions and $\lambda$-barycenters} \label{subsec:ShrinkingSubdivisions}

The results from the previous section lead us to two questions: How do we construct a push-off grid that is tight on the boundary and when can we make such a puch-off grid tight overall through the use of $\lambda$-shrinking subdivisions? In this section, we address the second question by providing a sufficient condition for the existence of $\lambda$-shrinking subdivisions. We start with the following definition.

\begin{mydef} \label{def:havingshrinkingsubdivisions}
    Let $(Z,d)$ be a metric space, $\lambda \in [\frac{1}{2}, 1)$, and $\Delta > 0$. We say that {\it $Z$ admits $\lambda$-shrinking subdivisions up to diameter $\Delta$} if for every finite simplicial complex $S$, every map $\iota : S^{(0)} \rightarrow Z$ with $\diam(\iota) \leq \Delta$, and every simplex $\sigma$ in $S$, we have the following property:
    
    Every $\lambda$-shrinking subdivision $(S', \iota')$ of $\partial \sigma$ that satisfies $\diam\left(\iota\left( S'^{(0)} \right) \right) \leq \diam\left( \iota\left( \sigma^{(0)} \right) \right)$ extends to a $\lambda$-shrinking subdivision $\sigma_{sub}$ of $\sigma$ with the property that for every simplex $\tilde{\sigma}$ in $S$, we have that
    \[ \diam\left( \iota \left( \tilde{\sigma}_{sub}^{(0)} \right) \right) \leq \diam \left( \iota \left( \tilde{\sigma}^{(0)} \right) \right), \]
    where $\tilde{\sigma}_{sub}^{(0)} = \tilde{\sigma}^{(0)} \cup ( \tilde{\sigma} \cap \sigma_{sub}^{(0)} )$ denotes the set of all vertices of $\tilde{\sigma}$ together with all vertices of $\sigma_{sub}$ that are contained in the (closed) simplex $\tilde{\sigma}$.

\end{mydef}

Given a finite or cocompact simplicial complex, the ability to subdivide any simplex allows us to subdivide all of them in finitely many steps. We have the following Lemma.

\begin{lem} \label{lem:admittingimpliesconstructionofshrinkingsubdivisions}
    Let $S$ be a locally finite simplicial complex and $(Z,d)$ a metric space, both with an isometric (simplicial) action by a group $H$, such that $S$ is $H$-cocompact. Suppose $Z$ admits $\lambda$-shrinking subdivisions up to diameter $\Delta$ and $\iota : S^{(0)} \rightarrow Z$ is an $H$-equivariant map with $\diam(\iota) \leq \Delta$. Then there exists a $\lambda$-shrinking subdivision $(S_{sub}, \iota_{sub})$ of $(S, \iota)$. 
\end{lem}

\begin{proof}
    Given $(S, \iota)$ as above, we need to construct a $\lambda$-shrinking subdivision. We do so by constructing the subdivision on the $l$-skeleton of $S$ and use induction over $l$. Our induction assumption is the following:

    \begin{IndAssum*}
        Let $l \geq 0$ and let $S' \subset S$ be an $H$-invariant subcomplex of $S$ that contains $S^{(l)}$ and has dimension $\leq l+1$. Let $S'_{sub}$ be a locally finite subdivision of $S'$ and $\iota' : S'^{(0)}_{sub} \rightarrow Z$ an extension of $\iota$ with the following two properties:

    \begin{itemize}

        \item If $\sigma'$ is a simplex in $S'_{sub}$ and $\sigma$ the least-dimensional simplex in $S'$ containing $\sigma'$, then
        \[ \diam\left( \iota' \left( \sigma'^{(0)} \right) \right) \leq \lambda \diam\left( \iota \left( \sigma^{(0)} \right) \right). \]

        \item For every simplex $\sigma$ in $S$, we have
        \[ \diam\left( \iota' \left( \sigma_{sub}^{(0)} \right) \right) \leq \diam\left( \iota \left( \sigma^{(0)} \right) \right), \]
        where $\sigma_{sub}^{(0)} = \sigma \cap S'^{(0)}_{sub}$ is the set of all vertices in $S'_{sub}$ contained in the (closed) simplex $\sigma$.
    \end{itemize}
    \end{IndAssum*}

    Clearly, $\iota$ satisfies these conditions for $l = 0$, which provides us with the start of the induction. For the induction step, let $\sigma$ be an $l+1$-dimensional simplex in $S$, which is not contained in $S'$. The induction step is complete if we can extend the subdivision and $\iota'$ to $\sigma$. Combining the fact that $\diam(\iota) \leq \Delta$ with the induction assumption, we obtain that
    \[ \diam\left( \iota' \left( \sigma_{sub}^{(0)} \right) \right) \leq \diam\left( \iota \left( \sigma^{(0)} \right) \right) \leq \Delta. \]
    Note that the entire boundary $\partial \sigma$ is contained in $S'$, since $S'$ contains the $l$-skeleton of $S$ by assumption. Since $Z$ admits $\lambda$-shrinking subdivisions up to diameter $\Delta$, we conclude that the subdivision of $\partial \sigma$ induced by $S'_{sub}$ extends to a subdivision $\sigma_{sub}$ of $\sigma$ and that we can extend $\iota'$ to this subdivision such that the following two inequalities hold: For all simplices $\sigma'$ in $S'_{sub} \cup \sigma_{sub}$, we have
    \[ \diam \left( \iota' \left( \sigma'^{(0)} \right) \right) \leq \lambda \diam \left( \iota \left( \sigma_{least}^{(0)} \right) \right), \]
    where $\sigma_{least}$ is the least-dimensional simplex in $S' \cup \sigma$ containing $\sigma'$. Furthermore, for all simplices $\tilde{\sigma}$ in $S$, we have
    \[ \diam \left( \iota' \left( \tilde{\sigma}_{sub}^{(0)} \right) \right) \leq \diam \left( \iota \left( \tilde{\sigma}^{(0)} \right) \right). \]
    Using the $H$-action on $S$, the subdivision of $\sigma$ and extension of $\iota'$ can be extended $H$-equivariantly to every simplex in $S$ of the form $h\sigma$ for some $h \in H$. Since $H$ acts cocompactly on $S$, we can subdivide the entire $l+1$-skeleton of $S$ and extend $\iota'$ to this subdivision in finitely many steps. This completes the induction step.

    Through this induction, we obtain a locally finite subdivision $S_{sub}$ of $S$ and an extension $\iota' : S_{sub}^{(0)} \rightarrow Z$. The properties in the induction assumption imply that $(S_{sub}, \iota')$ is a $\lambda$-shrinking subdivision, which proves the Lemma.

\end{proof}

We now provide a sufficient condition for $Z$ to admit $\lambda$-shrinking subdivisions. This condition is based on the idea that there is one particular subdivision of a simplicial complex that is very natural to consider: the barycentric subdivision.

\begin{mydef} \label{def:lambdaBarycenters}
Let $Z$ be a metric space, $\lambda \in [\frac{1}{2}, 1)$, and $P \subset Z$ a finite set. We call a point $b \in Z$ a {\it $\lambda$-barycenter} of $P$ if
\[ \forall p \in P : d(b, p) \leq \lambda \cdot \diam(P). \]
If, additionally, $Q \subset Z$ is a finite subset, we call $b \in Z$ a {\it $\lambda$-barycenter of $P$ relative to $Q$} if it is a $\lambda$-barycenter of $P$ and, additionally,
\[ \forall q \in Q : d(b, q) \leq \diam( \{ q \} \cup P). \]

Let $\Delta > 0$. We say that {\it $Z$ has $\lambda$-barycenters up to diameter $\Delta$}, if for any two finite sets $P, Q \subset Z$ such that $\diam(P) \leq \Delta$ and $\diam(P \cup Q) \leq 2\Delta$, there exists a $\lambda$-barycenter of $P$ relative to $Q$.


\end{mydef}

\begin{lem} \label{lem:barycentersimplyshrinkingsubdivisions}
    Let $\lambda \in [\frac{1}{2}, 1)$, $\delta \geq 0$, $Z$ a metric space, and $S$ a locally finite simplicial complex, both with an isometric (simplicial) action by a group $H$ such that $S$ is $H$-cocompact. Suppose $\iota : S^{(0)} \rightarrow Z$ is an $H$-equivariant map and suppose $Z$ has $\lambda$-barycenters up to diameter $\diam(\iota)$. Then there exists a $\lambda$-shrinking subdivision $(S', \iota_{sub})$ of $(S, \iota)$, where $S'$ is the barycentric subdivision of $S$.
\end{lem}

\begin{proof}
    We recall that we defined the room $Room(\sigma)$ around the simplex $\sigma$ in a simplicial complex to be the closure of the union of all simplices in $S$ that contain $\sigma$ as a face. Note that for every simplex,
    \[ \diam\left( \iota\left( Room( \sigma)^{(0)} \right) \right) \leq 2 \diam(\iota) \]
    as any two simplices in $Room(\sigma)$ have all vertices of $\sigma$ in common.
    
    If $Z$ has $\lambda$-barycenters, then it satisfies a variation of the property to admit $\lambda$-shrinking subdivisions, which allows us to mimic the proof of Lemma \ref{lem:admittingimpliesconstructionofshrinkingsubdivisions}. Namely, we begin with $\iota$ being defined on the vertices of a simplicial complex $S$. We can extend $\iota$ to the first barycentric subdivision of the $1$-skeleton of $S$ by choosing for every edge $(v,w)$ in $S$ a $\lambda$-barycenter of the pair $(\iota(v), \iota(w))$ relative to the set $\{ \iota( u ) \vert u \in Room( (v,w) ) \}$. In particular, if $H$ acts on $S$ and $\iota$ is $H$-equivariant on the vertices of $H$, we can choose these $\lambda$-barycenters in an $H$-equivariant way and extend $\iota$ to an $H$-equivariant map on the barycentric subdivision of the $1$-skeleton of $S$. This is a $\lambda$-shrinking subdivision of the $1$-skeleton of $S$.
    
    Assume now by induction that we have extended $\iota$ to a $\lambda$-shrinking subdivision of the barycentric subdivision of the $l$-skeleton of $S$. Denote the barycentric subdivision of the $l$-skeleton by $S'_l$. For any $l+1$-dimensional simplex in $S$, we can choose ($H$-equivariantly) a $\lambda$-barycenter of $S'_l \cap \partial \sigma$ relative to all vertices in $S'^{(0)}_l \cap Room(\sigma)$. This yields an $H$-equivariant, $\lambda$-shrinking subdivision of the $l+1$-skeleton of $S$, where the subdivision of $S^{(l+1)}$ is its barycentric subdivison. Induction over $l$ provides us with a $\lambda$-shrinking subdivision $(S', \iota_{sub})$ of $(S, \iota)$ such that $S'$ is the barycentric subdivision of $S$.
\end{proof}

\begin{figure} 
\begin{tikzpicture}[scale=1.25]

\draw [] (-5,2) -- (-5,-2);
\draw [] (-7, 0) -- (-5,2);
\draw [] (-7, 0) -- (-5,-2);
\draw [] (-5,2) -- (-3,0);
\draw [] (-5,-2) -- (-3, 0);
\draw [] (-7, 0) -- (-5.5,0);
\draw [] (-4.8,0) -- (-3, 0);
\draw [fill] (-5,2) circle [radius = 0.025cm];
\draw [fill] (-5,-2) circle [radius = 0.025cm];
\draw [fill] (-3,0) circle [radius = 0.025cm];
\draw [fill] (-7,0) circle [radius = 0.025cm];
\node [left] at (-7,0) {$q_1$};
\node [above left] at (-6,1) {$q_2$};
\node [below right] at (-4,-1) {$p_5$};
\node [right] at (-3,0) {$p_3$};
\node [above] at (-5,2) {$p_1$};
\node [below] at (-5,-2) {$p_2$};
\draw [thick] (-6,1) node[cross] {};
\draw [thick] (-4,-1) node[cross] {};
\draw [thick] (-5,0) node[cross] {};
\node [left] at (-5,0) {$p_4$};
\draw [thick] (-4,1) node[cross] {};
\node [above right] at (-4,1) {$p_6$};
\draw [] (-4.5,0.5) circle [radius=0.05cm];
\node [below] at (-4.5,0.5) {$v_{P}$};
\draw [thick] (-6,-1) node[cross] {};
\node [below left] at (-6,-1) {$q_3$};
\draw [thick] (-6,0) node[cross] {};
\node [below] at (-6,0) {$q_4$};

\draw [] (0,2) -- (0,-2);
\draw [] (-2, 0) -- (0,2);
\draw [] (-2, 0) -- (0,-2);
\draw [] (0,2) -- (2,0);
\draw [] (0,-2) -- (2, 0);
\draw [dotted] (-2, 0) -- (-0.5,0);
\draw [dotted] (0.2,0) -- (2, 0);
\draw [fill] (0,2) circle [radius = 0.025cm];
\draw [fill] (0,-2) circle [radius = 0.025cm];
\draw [fill] (2,0) circle [radius = 0.025cm];
\draw [fill] (-2,0) circle [radius = 0.025cm];
\node [left] at (-2,0) {$q_1$};
\node [above left] at (-1,1) {$q_2$};
\node [below right] at (1,-1) {$p_5$};
\node [right] at (2,0) {$p_3$};
\node [above] at (0,2) {$p_1$};
\node [below] at (0,-2) {$p_2$};
\draw [thick] (-1, -1) node[cross] {};
\node [below left] at (-1,-1) {$q_3$};
\draw [thick] (-1,0) node[cross] {};
\node [below] at (-1,0) {$q_4$};
\draw [thick] (0,0) node[cross] {};
\node [left] at (0,0) {$p_4$};
\draw [thick] (1,1) node[cross] {};
\node [above right] at (1,1) {$p_6$};
\draw [] (0.5,0.5) circle [radius=0.05cm];
\node [right] at (0.5, 0.56) {$\iota(v_P)$};
\draw [dashed] (-2,0) -- (0.5,0.5);
\draw [] (2,0) -- (0.5,0.5);
\draw [dashed] (-1,1) -- (0.5,0.5);
\draw [] (1,-1) -- (0.5,0.5);
\draw [] (0,2) -- (0.5,0.5);
\draw [] (0,-2) -- (0.5,0.5);
\draw [] (0,0) -- (0.5,0.5);
\draw [] (1,1) -- (0.5,0.5);
\draw [dashed] (-1, -1) -- (0.5,0.5);
\draw [dashed] (-1, 0) -- (0.5, 0.5);
\end{tikzpicture}
\caption{When extending $\iota$ to the barycentric subdivision, we need to define it for all barycenters, like $v_P$ in this figure (left), which is the barycenter of $p_1, p_2, p_3$. By induction, we assume $\iota$ has already been extended to the barycentric subdivision of the $1$-skeleton. We then define $\iota(v_P)$ to be a $\lambda$-barycenter of $p_1, \dots, p_6$ relative to $q_1, \dots q_4$.}
\label{fig:BarycenterExtension} 
\end{figure}
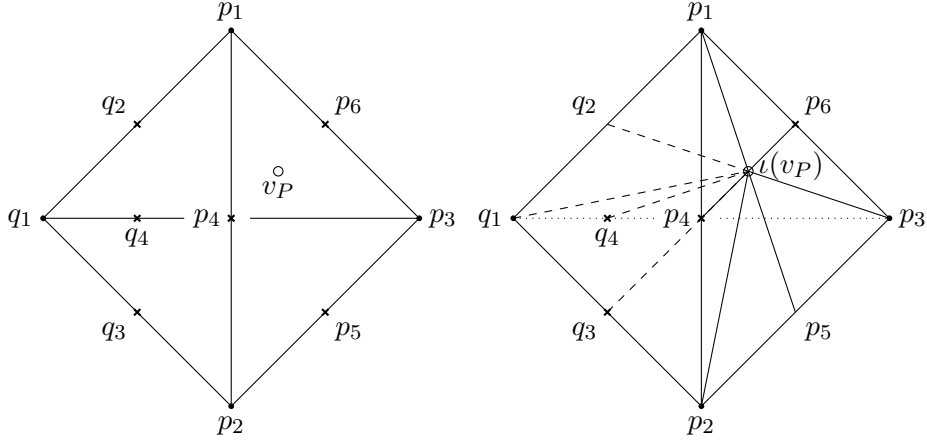

We need the following Lemma in order to control $\diam(\iota)$.

\begin{lem} \label{lem:diameterofiota}
    Let $K \subset C_{\epsilon}$ be a compact set such that $HK = C_{\epsilon}$, $K_{out} \subset Y_{\epsilon}$ a compact set, $Z := HK_{out}$, and $(\delta, \delta')$ small relative to $K$, $K_{out}$, and some $\alpha > \frac{\pi}{2}$.
    
    If $\mathcal{U}$ is a finite cover of $K$ by open sets and $\iota : N(H\mathcal{U})^{(0)} \rightarrow Z$ is an $H$-equivariant map satisfying $\iota(\mathcal{U}) \subset K_{out}$, then $\diam(\iota) \leq \diam_K(K_{out})$.
\end{lem}

\begin{proof}
    Let $U, U'$ be two adjacent vertices in $N(H\mathcal{U})$, i.e.\,$U \cap U' \neq \emptyset$. We can assume without loss of generality that $U \in \mathcal{U}$. If $U' \in \mathcal{U}$, then $\iota(U), \iota(U') \in K_{out}$ and $d(\iota(U), \iota(U')) \leq \diam(K_{out}) \leq \diam_K(K_{out})$. Otherwise, there exists $h \in H$ and $U'' \in \mathcal{U}$ such that $U' = hU''$ and thus $hU'' \cap U \neq \emptyset$. Since $\diam(U), \diam(U') \leq \delta$, this implies that there exist points $p \in K \cap U$ and $q \in hK \cap U'$ such that $d(p,q) \leq 2\delta$. Since $(\delta, \delta')$ are small relative to $K$, $K_{out}$, and $\pi$, this implies that $hK \cap K \neq \emptyset$ and thus
    \[ \diam_K(K_{out}) \geq \diam(hK_{out} \cup K_{out}) \geq d(\iota(U), \iota(U')). \]
\end{proof}

Let $Z \subset X$ be an $H$-invariant subset satisfying $HK_{out} \subset Z \subset Y_{\epsilon + \delta'}$. If $Z$ has $\lambda$-barycenters, we can now combine Lemma \ref{lem:barycentersimplyshrinkingsubdivisions} with Theorem \ref{thm:Tightpushoffgridsimplycontractibility} to obtain a sufficient condition for contractibility of $\Sigma_{\epsilon}$. We capture the sufficient condition in the following definition. 

\begin{mydef} \label{def:barycentersuptoH}
    Let $\epsilon > 0$, $X$ be a $\CAT$ space, $C \subset X$ a closed, convex, $H$-invariant and $H$-cocompact subset, $\Sigma_{\epsilon}$ an $H$-invariant boundary component of $C_{\epsilon}$ and $Y_{\epsilon} \subset X \setminus C_{\epsilon}$ the connected component bounded by $\Sigma_{\epsilon}$.

    Let $\lambda \in [\frac{1}{2},1)$. We say that $C_{\epsilon}$ has {\it $\lambda$-barycenters up to $H$}, if for some
    
    \begin{itemize}
        \item compact set $K \subset C_{\epsilon}$ such that $HK = C_{\epsilon}$,

        \item compact set $K_{out} \subset Y_{\epsilon}$,

        \item $\alpha > \frac{\pi}{2}$,

        \item pair $(\delta, \delta')$ that is small relative to $K$, $K_{out}$, and $\alpha$,

        \item finite cover $\mathcal{U}$ of $K$ by open sets with diameter at most $\delta$,

    \end{itemize}

    there exists

    \begin{enumerate}

        \item an $H$-invariant subset $Z \subset X$ that has $\lambda$-barycenters up to diameter $\diam_K(K_{out})$ and satisfies $HK_{out} \subset Z \subset Y_{\epsilon + \delta'}$,

        \item a $(\mathcal{U}, \alpha)$-push-out grid $\iota : H\mathcal{U} \rightarrow HK_{out}$ that is $(\delta, \delta')$-tight on the boundary and satisfies $\iota(\mathcal{U}) \subset K_{out}$.        
    \end{enumerate}
    
\end{mydef}

We recall that, based on the constructions of this section, the main challenge is to obtain $K_{out}$, $(\delta, \delta')$, $Z$, and $\iota$. The condition that $Z \subset Y_{\epsilon + \delta'}$ is required to make sure that $d_{push-off}(\iota_n) > \delta'$, which is needed in order to apply Theorem \ref{thm:Tightpushoffgridsimplycontractibility}. By Lemma \ref{lem:diameterofiota}, $\diam(\iota) \leq \diam_K(K_{out})$. We obtain the following result.

\begin{thm} \label{thm:lambdabarycentersimplycontractibility}
    Let $\epsilon > 0$, $X$ be a $\CAT$ space, $C \subset X$ be closed, convex, $H$-invariant, and $H$-cocompact, $\Sigma_{\epsilon}$ an $H$-invariant boundary component of $C_{\epsilon}$ and $Y_{\epsilon} \subset X \setminus C_{\epsilon}$ be the connected component bounded by $\Sigma_{\epsilon}$.

    If there exists some $\lambda \in [\frac{1}{2}, 1)$ such that $C_{\epsilon}$ has $\lambda$-barycenters up to $H$, then $\Sigma_{\epsilon}$ is contractible.



\end{thm}








\section{Applications to limit sets} \label{sec:applications}

\subsection{Constructing push-off grids in manifolds} \label{sec:constructingpushoffgridsinmanifolds}

In this section, we address the other assumption of Theorem \ref{thm:Tightpushoffgridsimplycontractibility}: the construction of a push-off grid that is $(\delta, \delta')$-tight on the boundary for small $\delta, \delta'$. We do so under some extra assumptions that will provide us with results regarding our initial motivating questions.

Suppose $X$ is a simply connected, geodesically complete Riemannian manifold with sectional curvature in $[-b^2, -1]$. Let $H$ be a group acting properly, freely, and by isometries on $X$ and $C \subset X$ be an $H$-invariant, $H$-cocompact convex subset. For every $\epsilon > 0$, the set $C_{\epsilon}$ is 
$H$-invariant and, by \cite{Walter76}, the topological boundary of $C_{\epsilon}$ is a $C^{1,1}$-manifold. Let $\Sigma_{\epsilon}$ be a connected component of $\partial C_{\epsilon}$ that is preserved by $H$ and $Y_{\epsilon} \subset X \setminus C_{\epsilon}$ the connected component bounded by $\Sigma_{\epsilon}$. (We note that there may be no $H$-invariant component. If $X \setminus C$ has only finitely many connected components, then we can replace $H$ by a finite index subgroup, which still acts cocompactly on $C$ and preserves all boundary components.)

For any $\epsilon > 0$, we can consider the differentiable function $x \mapsto d(x, C_{\epsilon})$. The gradient of this function is a continuous, $H$-invariant vector field, which we denote by $N$. This vector field is normal to $\Sigma_{R}$, pointing out of $C_R$ for every $R > \epsilon$. Furthermore, if $R > \epsilon > \epsilon'$, then the vector fields induced by $\epsilon$ and $\epsilon'$ coincide on $\Sigma_R$. We can thus let $\epsilon$ go to zero and obtain a continuous, $H$-invariant vector field $N$ on $Y_{\epsilon}$ that is normal to $\Sigma_R$ for all $R > 0$.

We denote the flow along $N$ for time $t$ by $\Phi_N^t : Y_{\epsilon} \rightarrow Y_{\epsilon}$. This is a $C^1$-diffeomorphism with the property that $\Phi_N^t( \Sigma_{\epsilon} ) = \Sigma_{\epsilon + t}$. Since $H$ acts cocompactly on $\Sigma_{\epsilon}$ and $\Sigma_R$, the restriction of $\Phi_N$ to $\Sigma_{\epsilon}$ is uniformly continuous. In particular, for every $\delta' > 0$, there exists $\delta > 0$, depending on $\delta', \epsilon$, and $R$, such that if $p, q \in \Sigma_{\epsilon}$ satisfy $d( p, q ) \leq 2\delta$, then $d( \Phi_N^{R-\epsilon}(p), \Phi_N^{R-\epsilon}(q) ) \leq \delta'$.\\

\begin{lem} \label{lem:appl:constructingaboundarytightpushoffgrid}
    Let $R > \epsilon > 0$. For all $\delta' > 0$ sufficiently small, there exists $\delta > 0$ and an $H$-invariant cover $H\mathcal{U}$ of $C_{\epsilon}$ so that there exists a $(\mathcal{U}, \pi)$-push-off grid $\iota : N( H\mathcal{U} )^{(0)} \rightarrow \Sigma_R$ that is $(\delta, \delta')$-tight on the boundary.
\end{lem}

\begin{proof}
    Fix $R > \epsilon > 0$. Let $K \subset C_{\epsilon}$ be a compact set such that $HK = C_{\epsilon}$. Put
    \[ K_{out} := \Phi_N^{R-\epsilon} ( K \cap \Sigma_{\epsilon}). \]
    Let $\delta' > 0$ be sufficiently small so that there exists $\delta > 0$ such that $(\delta, \delta')$ are small relative to $K$, $K_{out}$, and $\pi$. Shrinking $\delta$ if necessary, we can make it so that for all $p, q \in \Sigma_{\epsilon}$ that satisfy $d( p, q ) \leq 2\delta$, we have $d( \Phi_N^{R-\epsilon}(p), \Phi_N^{R-\epsilon}(q) ) \leq \delta'$. Furthermore, we can shrink $\delta$ even further if necessary, so that every cover of $C_{\epsilon}$ by sets of diameter $\leq \delta$, is $H$-fine. We will construct a suitable push-off grid for all $\delta, \delta'$ with these properties.

    Let $\mathcal{U}$ be any finite cover of $K$ by open sets in $X$ that have diameter $\leq \delta$. (We assume that every element in $\mathcal{U}$ has non-empty intersection with $K$.) We conclude that $H\mathcal{U}$ is a locally finite, $H$-invariant, $H$-fine cover of $C_{\epsilon}$. We define
    \[ \mathcal{U}_{\Sigma} := \{ U \in \mathcal{U} \vert U \cap \Sigma_{\epsilon} \neq \emptyset \}. \]
    We now define $\iota : N( H\mathcal{U} )^{(0)} \rightarrow K_{out}$ as follows: Let $U \in \mathcal{U}$. If $U \notin \mathcal{U}_{\Sigma}$, then we choose some point $q' \in K_{out}$ and put $\iota(U) := q'$. If $U \in \mathcal{U}_{\Sigma}$, then there exists some $q \in U \cap \Sigma_{\epsilon}$. We choose one such $q$ and define $\iota(U) := \Phi_N^{R-\epsilon}(q)$. We can extend this map $H$-equivariantly to $H\mathcal{U}$. Since $H\mathcal{U}$ is $H$-fine this extension is well-defined and we obtain an $H$-equivariant map $\iota : H\mathcal{U} \rightarrow \Sigma_R$ with the property that $\iota( \mathcal{U} ) \subset K_{out}$.\\

    We claim that $\iota$ is a $(\mathcal{U}, \pi)$-push-off grid that is $(\delta, \delta')$-tight on the boundary. We first show that it is a push-off grid. For this, we have to show that for all $U \in \mathcal{U}_{\Sigma}$, there exists some $q \in U$ such that
    \[ \angle_q(\iota(U), C) \geq \pi. \]
    By construction, there exists some $q \in U \cap \Sigma_{\epsilon}$ such that $\iota(U) = \Phi_N^{R-\epsilon}(q)$. This $q$ satisfies
    \[ \angle_q(\iota(U), C) = \angle_q( \Phi_N^{R-\epsilon}(q), C) = \pi. \]
    Thus, $\iota$ is a $(\mathcal{U}, \pi)$-push-off grid.

    We are left to show that $\iota$ is $(\delta, \delta')$-tight on the boundary. We know that every element of $H\mathcal{U}$ has diameter $\leq \delta$ since this holds for all elements of $\mathcal{U}$ by assumption and $H$ acts by isometries. To estimate the diameter of $\iota$ on the subcomplex $N(H\mathcal{U}_{\Sigma})$, we consider two elements $U, U' \in H\mathcal{U}_{\Sigma}$ that are connected by an edge. Let $q \in U, q' \in U'$ such that $\iota(U) = \Phi_N^{R-\epsilon}(q)$ and $\iota(U') = \Phi_N^{R-\epsilon}(q')$. Since $U$ and $U'$ are connected by an edge in $N(H\mathcal{U})$, we know that $U \cap U' \neq \emptyset$ and thus, $d(q, q') \leq 2\delta$. We estimate
    \[ d( \iota(U), \iota(U') ) = d( \Phi_N^{R-\epsilon}(q), \Phi_N^{R-\epsilon}(q') ) \leq \delta', \]
    as $\delta, \delta'$ were chosen to satisfy the uniform continuity-property. We conclude that $\iota$ is $(\delta, \delta')$-tight on the boundary.
    
\end{proof}



Lemma \ref{lem:appl:constructingaboundarytightpushoffgrid} tells us that we can reliably construct push-off grids that are $(\delta, \delta')$-tight on the boundary for arbitrarily small $\delta, \delta'$. We conclude that, for manifolds, contractibility of $\Sigma_{\epsilon}$ is only a matter of the existence of barycenters in $Y_{\epsilon}$. To state the result, we briefly recall our notation.

\begin{notation}
    Let $X$ be a Cartan-Hadamard manifold with sectional curvature in $[-b^2, -1]$, $H$ a group acting properly, freely, and by isometries on $H$, $C \subset X$ a closed, $H$-invariant, $H$-cocompact convex subset, $\epsilon > 0$, $\Sigma_{\epsilon}$ an $H$-invariant connected component of $\partial C_{\epsilon}$ and $Y_{\epsilon}$ the connected component of $X \setminus C_{\epsilon}$ bounding $\Sigma_{\epsilon}$.

    Let $K \subset C_{\epsilon}$ be a compact set such that $HK = C_{\epsilon}$ and $K_{out} := \Phi_N^{R-\epsilon}( K \cap \Sigma_{\epsilon} )$ for every $R > \epsilon$.
\end{notation}

\begin{thm} \label{thm:appl:barycentersimplycontractibility}
    If there exist $\lambda \in [\frac{1}{2},1)$, $R > \epsilon$, and an $H$-invariant set $\Sigma_R \subset Z \subset Y_{\epsilon + \delta'}$ such that $Z$ has $\lambda$-barycenters up to diameter $\diam_K(K_{out})$, then $\Sigma_{\epsilon}$ is contractible.
\end{thm}

\begin{proof}
    Choose $\lambda$, $R$, and $Z$ as in the assumptions of the theorem. By Lemma \ref{lem:appl:constructingaboundarytightpushoffgrid}, there exist $\delta, \delta'$ that are small relative to $K$, $K_{out}$ and $\pi$, a finite cover $\mathcal{U}$ of $K$ by sets of diameter $\leq \delta$ and a $(\mathcal{U}, \pi)$-push-off grid $\iota$ such that $\iota(\mathcal{U}) \subset K_{out} \subset Z$ and $\iota$ is $(\delta, \delta')$-tight on the boundary. Since $Z$ has $\lambda$-barycenters up to diameter $\diam_K(K_{out})$ and $\diam(\iota) \leq \diam_K(K_{out})$ by Lemma \ref{lem:diameterofiota}, Theorem \ref{thm:lambdabarycentersimplycontractibility} implies that $\Sigma_{\epsilon}$ is contractible.
\end{proof}




\subsection{Contractibility in the visual boundary} \label{subsec:contractibilityatinfinity}

Let $\Lambda$ be the limit set of the action of $H$ on $X$. If $C$ is the closed convex hull of $\Lambda$, Theorem \ref{thm:Andersonresult} by Anderson implies that $\partial_{\infty} C = \Lambda$. As discussed earlier, we have a homeomorphism
\[ \Phi_N^{\infty} : \partial C_{\epsilon} \rightarrow \partial_{\infty} X \setminus \Lambda \]
Thus, Theorem \ref{thm:appl:barycentersimplycontractibility} provides us a sufficient condition for contractibility of the connected components of $\partial_{\infty} X \setminus \Lambda$.

Let $Z$ be a connected component of $\partial_{\infty} X \setminus \Lambda$. We would like to have a variation of our result that uses an assumption on $Z$ instead of some $\Sigma_R$. Taking an overview of our proof, we see that our overall strategy relied on the following points: We construct a push-off grid that is $(\delta, \delta')$-tight. If we can do so for small $\delta, \delta'$, we can make sure that our extension of the grid to an actual push-off does not change angles along $\Sigma_{\epsilon}$ too much and that the convex hulls that we use in our extension do not intersect $C_{\epsilon}$. (These are two of the three properties in Definition \ref{def:smalldeltarelativetoR} that determine that $(\delta, \delta')$ are small relative to $K$, $K_{out}$, and $\alpha$.) Our entire proof works in the same way for a push-off grid $\iota : H\mathcal{U} \rightarrow Z$, provided that, given any angle $\alpha > \frac{\pi}{2}$ and any compact sets $K \subset C_{\epsilon}$ and $K_{out} \subset Z$, we can find small $(\delta, \delta')$ that satisfy a variation of properties (1), (2), and (3) of Definition \ref{def:smalldeltarelativetoR}. The following two statements tell us that this can be done.

\begin{lem}
    The function $(o,p,q) \mapsto \angle_o(p,q)$ continuously extends to a function
    \[ X \times \overline{X} \times \overline{X} \rightarrow [0,\pi] \]
    \[ (o,p,q) \mapsto \angle_o(p,q), \]
    where $\angle_o(p,q)$ is the angle between the geodesics from $o$ to $p$ and $o$ to $q$.
\end{lem}

This is a standard result, which tells us that for $\delta, \delta'$ sufficiently small, property (3) of Definition \ref{def:smalldeltarelativetoR} is satisfied, even if $K_{out}$ is a compact subset of $Z$ instead of $Y_{\epsilon}$. The next result allows us to find a replacement for property (2).

\begin{lem} \label{lem:SmallDiameteratInfinity}
Let $o \in X$, $C \subset X$ a closed and convex subset, $K_{out} \subset \partial_{\infty} X \setminus \partial_{\infty} C$ compact. Then there exists $\delta' > 0$ such that for all sets $P \subset K$ with diameter at most $\delta'$, $C(P) \cap C = \emptyset$.
\end{lem}

\begin{proof}
The key tool of the proof is {\cite[Lemma 2.6]{Bowditch94}}, which states that the convex hull of a closed set $P \in X \cup \partial_{\infty} X$ is contained in the $r$-neighbourhood of the set of all geodesics between points in $P$, where $r$ depends only on the lower curvature bound of $X$. Let $r$ denote the constant $\sigma$ in {\cite[Lemma 2.6]{Bowditch94}} and let $\delta > 0$ such that $X$ is $\delta$-hyperbolic. Since $K_{out}$ is compact, there exists a constant $T > 0$, such that for all geodesic rays $\gamma$ that start at $o$ and represent points in $K_{out}$, we have that $\gamma\vert_{[T,\infty)} \cap N_{r+3\delta}(C) = \emptyset$. Choose $\delta'$ sufficiently small, such that whenever $\rho_o(\xi, \tilde{\xi}) \leq \delta'$ and $\gamma, \tilde{\gamma}$ are the unique geodesic rays starting at $o$ representing $\xi$ and $\tilde{\xi}$ respectively, then $d(\gamma(T), \tilde{\gamma}(T)) \leq \delta$.

Let $P \subset K_{out}$ be a set of diameter at most $\delta'$ with respect to $\rho_o$. Replace $P$ by its closure, which still satisfies the same bound on its diameter. Define
\[ \Gamma_{o,T} := \bigcup_{\gamma : \gamma(0) = o, \gamma(\infty) \in P } \gamma\vert_{[T,\infty)}, \]
i.e.\,$\Gamma_{o,T}$ considers all geodesic rays starting at $o$ (at time zero) that represent points in $P$ and restricts them to the interval $[T, \infty)$. Fix some $\xi \in P$ and let $\gamma$ be its geodesic representative starting at $o$. We obtain
\[ \Gamma_{\gamma(T), 0} := \bigcup_{ \gamma' : \gamma'(0) = \gamma(T), \gamma(\infty) \in P} \gamma', \]
which consists of all geodesic rays starting at $\gamma(T)$ representing points in $P$. Since $X$ has non-positive curvature, $\Gamma_{\gamma(T), 0}$ is contained in the $\delta$-neighbourhood of $\Gamma_{o,T}$. By $\delta$-hyperbolicity, we see that the union of all bi-infinite geodesics between points in $P$ is contained in the $\delta$-neighbourhood of $\Gamma_{\gamma(T), 0}$. By Bowditch, the convex hull $C(P)$ of $P$ is contained in the $r$-neighbourhood of the union of these bi-infinite geodesics. We conclude that $C(P)$ is contained in the $r + 2\delta$-neighbourhood of $\Gamma_{o,T}$. Due to our choice of $T$, this neighbourhood does not intersect $C$. We conclude that $C(P) \cap C = \emptyset$, which proves the Lemma.
\end{proof}

Looking at the proof of Lemma \ref{lem:extendingpushoffgridtopushoff}, we see that condition (2) of Definition \ref{def:smalldeltarelativetoR} can be replaced with the condition that for every simplex $\sigma$ in the subdivision $S_n$, we have that
\[ C( \iota_n( \sigma^{(0)} ) ) \cap C_{\epsilon} = \emptyset. \]
This allows us to make the following definition.

\begin{mydef}
    Let $\alpha > \frac{\pi}{2}$, $K \subset C_{\epsilon}$ be compact such that $HK = C_{\epsilon}$, $K_{out} \subset Z$ be compact, and $\rho$ a fixed visual metric on $\partial_{\infty} X$. We say that $(\delta, \delta')$ are small relative to $K$, $K_{out}$, and $\alpha$ when the following two properties hold:
    \begin{enumerate}
        \item For all finite subsets $P \subset K_{out}$ with $\diam(P) \leq \delta'$, we have $C(P) \cap C_{\epsilon} = \emptyset$.

        \item For all $q_1, q_2 \in N_{\delta}(K) \cap \Sigma_{\epsilon}$ and all $q'_1, q'_2 \in K_{out}$, we have that
        \[ d(q_1, q_2) \leq \delta, \rho(q'_1, q'_2) \leq \delta' \Rightarrow \vert \angle_{q_1}(q'_1, C) - \angle_{q_2}(q'_2, C) \vert \leq \frac{\alpha}{2} - \frac{\pi}{4}. \]
    \end{enumerate}
\end{mydef}

Lemma \ref{lem:SmallDiameteratInfinity} tells us that for any map $\iota : H\mathcal{U} \rightarrow Z$, there exists some $\delta, \delta' > 0$ that are small relative to $K, \iota( \adj( \mathcal{U} ) )$, and $\alpha$. We can then generalize the definition of a push-off to include maps $j : N(H\mathcal{U}) \rightarrow Z$ and adjust the proof of Lemma \ref{lem:extendingpushoffgridtopushoff} to prove the following result.

\begin{lem} \label{lem:extendingtopushoffonboundary}
    Let $\lambda \in [\frac{1}{2},1)$, $\alpha > \frac{\pi}{2}$, $\mathcal{U}$ be a finite, $H$-fine cover of a compact set $K$ such that $HK = C_{\epsilon}$, and $\iota : H\mathcal{U} \rightarrow Z$ a $(\mathcal{U}, \alpha)$-push-off grid which is $(\delta, (1-\lambda) \delta' )$-tight on the boundary.
    
    Suppose $(\delta, \delta')$ is small relative to $K$, $\iota(\adj(\mathcal{U}) )$, and $\alpha$ and suppose $\iota$ admits an $H$-equivariant $\lambda$-shrinking subdivision $(S_n, \iota_n)$ of order $n$ which is $(\delta, \delta')$-tight, has push-off distance $> \delta'$, and satisfies $\iota_n : S_n^{(0)} \rightarrow Z$. Then $\iota_n$ can be extended to a $\mathcal{U}$-push-off through $\Sigma_{\epsilon}$.
\end{lem}

With these modifications, Theorem \ref{thm:appl:barycentersimplycontractibility} turns into the following Theorem.

\begin{thm} \label{thm:appl:contractibilityinboundary}
    Let $X$ be a Cartan-Hadamard manifold with sectional curvature in $[-b^2, -1]$, $H$ be a group acting properly, freely, and by isometries on $X$, $C \subset X$ be a closed, $H$-invariant, $H$-cocompact convex subset, $\epsilon > 0$, $\Sigma_{\epsilon}$, an $H$-invariant connected component of $\partial C_{\epsilon}$, and $Z$ the connected component of $\partial_{\infty} X \setminus \partial_{\infty} C$ corresponding to $\Sigma_{\epsilon}$.

    Let $K \subset C_{\epsilon}$ be a compact set such that $HK = C_{\epsilon}$ and $K_{out} := \Phi_N^{\infty}( K \cap \Sigma_{\epsilon} )$, where $\Phi_n^{\infty} = \lim_{t \rightarrow \infty} \Phi_n^t$.

    If there exists $\lambda \in [\frac{1}{2}, 1)$ such that $Z$ has $\lambda$-barycenters up to diameter $\diam_K(K_{out})$, then $Z$ is contractible.
\end{thm}



\bibliographystyle{alpha}
\bibliography{Codim1bib}

\end{document}